%% file: spectrumconsistency_arxiv.tex
\documentclass[letterpaper, 10 pt, journal]{ieeetran}
\usepackage{cite}
\usepackage{amsmath,amssymb,amsfonts}
\usepackage{algorithmic}
\usepackage{graphicx}
\usepackage{textcomp}
\def\BibTeX{{\rm B\kern-.05em{\sc i\kern-.025em b}\kern-.08em
    T\kern-.1667em\lower.7ex\hbox{E}\kern-.125emX}}

\usepackage{amsmath,amsfonts,amssymb}
\usepackage{pdfpages}

\usepackage[bookmarks=false,colorlinks=true,citecolor=blue,linkcolor=blue]{hyperref}
\usepackage{mathtools}
\usepackage{bbm}
\usepackage{epsfig}
\usepackage{times}
\usepackage{color}
\usepackage[english]{babel}
\usepackage{graphicx}
\usepackage{subfigure}
\usepackage{standalone}
\usepackage{tikz}
\usepackage{nicefrac}
\usepackage{comment}

\def\argmin{\operatornamewithlimits{arg\,min}}
\def \interpolates{\cong}

\newcommand{\defeq}{\doteq}

\newcommand{\gramnplus}{K_{0}}

\newcommand{\newpart}{}

\newcommand{\BX}{\mathbb{X}}

\newcommand{\BC}{\mathbb{C}}

\newcommand{\CC}{\mathcal{C}}
\newcommand{\CS}{\mathcal{S}}
\newcommand{\CF}{\mathcal{F}}

\newcommand{\CB}{\mathcal{B}}

\newcommand{\CL}{\mathcal{L}}

\newcommand{\CH}{\mathcal{H}}

\newcommand{\ConfEll}{\mathcal{E}_{n_0}^n}
\newcommand{\ConfEllStar}{\mathcal{E}_{n_0}^*}
\newcommand{\ConfHyp}{\mathcal{R}_{n_0}^n}

\newcommand\econvergealmostsurely{\xrightarrow[n\to\infty]{\mbox{\normalfont\tiny a.s.}}}
\newcommand\convergealmostsurely{\mathrel{\stackrel{\makebox[0pt]{\mbox{\normalfont\tiny a.s.}}}{\longrightarrow}}}

\makeatletter
\newcommand*{\rom}[1]{\expandafter\@slowromancap\romannumeral #1@}
\makeatother
\def \CG{\mathcal{G}}
\def \CF{\mathcal{F}}

\def \CZ{\mathcal{Z}}

\definecolor{MyBlue}{rgb}{0.2,0.2,0.64}

\newcommand{\tr}{^\mathrm{T}}
\newcommand{\RR}{\mathbb{R}}

\def\argmin{\operatornamewithlimits{arg\,min}}

\newtheorem{assumption}{\bf A\!\!}
\newtheorem{theorem}{Theorem}
\newtheorem{corollary}{Corollary}
\newtheorem{lemma}{Lemma}

\newtheorem{remark}{Remark}
\newenvironment{proof}{\noindent{\em Proof.\,}}{\hfill$\Box$\\}

\newcommand{\norm}[1]{\left\lVert#1\right\rVert}    
    
\newlength{\dhatheight}

\graphicspath{ {./Noise-free output images/} {./Noisy output images/} {./Spectra images/} }

\begin{document}

\title{\LARGE \bf MiNCE: Nonparametric, Strongly Consistent Confidence Envelopes for Band-Limited Functions and their Smoothed Spectra}
\author{Bal{\'a}zs Csan{\'a}d Cs{\'a}ji, \IEEEmembership{Senior Member, IEEE}, \qquad\qquad \and B\'alint Horv\'ath
\thanks{This research was partially supported by the ``Robust Uncertainty Quantification for Learning and Control'' ADVANCED project of the National Research, Development and Innovation Office of Hungary (NKFIH), grant number 153390. The work was also supported by the European Commission through the DiGreeS project under grant number 101178079.
}%
\thanks{\newpart B.~Cs.~Cs\'aji is with HUN-REN SZTAKI: Institute for Computer Science and Control, Budapest, Hungary; and also with Institute of Mathematics, E\"otv\"os Lor\'and University (ELTE), Budapest, Hungary, {\tt\small csaji@sztaki.hu}}%
\thanks{\newpart B.~Horv\'ath is with HUN-REN SZTAKI: Institute for Computer Science and Control, 
Budapest, Hungary,
{\tt\small balint.horvath@sztaki.hu}}%
}

\hyphenation{pa-ram-et-ri-za-ti-on}

\maketitle
\thispagestyle{plain}
\pagestyle{plain}

\begin{abstract}
Minimum-norm confidence envelope strategies offer a nonparametric approach to constructing nonasymptotic, simultaneous confidence regions for band-limited functions, exploiting the theory of Reproducing Kernel Hilbert Spaces (RKHS). While the finite-sample coverage guarantees of these envelopes have been established, their consistency has not been analyzed so far. In this paper, we study this construction, here termed the Minimum-Norm Confidence Envelope (MiNCE) framework, and establish the strong uniform consistency of the resulting bands, both for noise-free and noisy observation models, under mild assumptions on the measurement noises. We further extend this formulation to the frequency domain, deriving nonasymptotic, simultaneous, strongly uniformly consistent confidence bands for the smoothed spectra. Numerical experiments in nonparametric regression and spectral estimation empirically confirm our theoretical results, illustrating the contraction of the confidence envelopes toward the target function as the sample size increases. 
\end{abstract}


\section{Introduction}
\IEEEPARstart{O}{ne} of the fundamental problems in signal processing, system identification, adaptive control, machine learning, econometrics, and statistics is {\em regression}: given a finite sample of input-output measurements, we need to estimate an underlying {\em regression function} encoding the conditional expectation of the output given the input \cite{cucker2007learning}. Standard approaches, including linear and polynomial regression, splines, decision trees, kernel methods, and (deep) neural networks, typically provide {\em point estimates} by selecting the best model from a hypothesis class with respect to a given criterion \cite{gyorfi2002distribution}. Point estimates, however, are often insufficient for safety-critical applications and robust decision-making, where explicit uncertainty quantification via guaranteed {\em region estimates} is required.

Region estimation has several distinct problem formulations. Predicting 
the possible values of a
new (noisy) observation can be addressed by distribution-free methods such as
Interval Predictor Models (IPMs) based on the scenario approach
\cite{campi2009interval, garatti2019class} or Conformal Prediction (CP)
\cite{lei2014distribution}.  For confidence regions for the true data-generating function 
in the {\em parametric} setting, distribution-free approaches based on the Sign-Perturbed
Sums (SPS) method \cite{csaji2014sign} provide exact, user-chosen finite-sample
coverage \cite{Algo2018}. Alternatively, Gaussian Process (GP) regression
\cite{quinonero2005unifying, Rasmussen2006} offers a popular Bayesian framework
for constructing credible regions, though it relies on joint Gaussianity
that can be restrictive in practice. Strong distributional
assumptions can be relaxed, e.g., by quantile regression
\cite{koenker2001quantile}. In the nonparametric setting, nonasymptotic
confidence bands, assuming noise-free outputs, were recently built in Sobolev spaces
\cite{gamboa2025nonasympconfrkhs}.

In this paper, we study {\em nonparametric}, {\em simultaneous} confidence bands with {\em nonasymptotic} coverage for {\em band-limited} functions \cite{wingham2002reconstruction}. This class plays a foundational role in signal processing, underlying signal reconstruction, pulse shaping in digital communications, remote sensing, and frequency-division multiplexing. For parametric models, a confidence set in the parameter space induces a confidence set in the 
model
space; such {\em indirect} constructions, however, do not carry over to nonparametric
models, for which {\em direct} constructions are needed. Here, we build on our prior 
kernel-based interpolation strategies \cite{csaji2022nonparametric, csaji2023improving, horvath2023nonparametric, 11052262, horvath2025single}, 
which we call the {\em Minimum-Norm Confidence Envelope} (MiNCE) framework. While {\em nonasymptotic} coverage guarantees for these envelopes are available, e.g., \cite[Theorems 1 and 2]{csaji2022nonparametric}, their {\em consistency} has not been established.

The main contributions of the paper are as follows:
\vspace{1mm}
\begin{enumerate}
    \item {\em Strong Uniform Consistency}: We prove, assuming a known sampling distribution, that the MiNCE confidence regions are strongly uniformly consistent, both in the noise-free case and, under mild conditions on the observation noise, in the noisy case, as well. We establish the strong consistency of the abstract MiNCE regions in the function space and also that of the induced confidence bands, uniformly across all inputs. A consistent norm-ball outer approximation is also provided.\vspace{1mm}
    \item {\em Confidence Tubes for Smoothed Spectra}: We extend the MiNCE framework from spatial-domain regression to frequency-domain analysis, constructing nonasymptotic, nonparametric, simultaneous confidence tubes for the smoothed spectral density function, whose fibers are complex disks. We show that the MiNCE spectral confidence tubes have the same consistency guarantees as the spatial-domain bands, providing a mathematically sound foundation for safe, nonparametric spectral inference.
\end{enumerate}
\vspace{1.5mm}

We further validate our theoretical results numerically, in nonparametric regression and spectral estimation, illustrating the systematic, uniform contraction of the confidence envelopes toward the target function as the sample size increases.

\section{Kernels and Band-Limited Functions}
{\newpart Kernel methods have an immense range of applications in statistics, signal processing, system identification and machine learning \cite{pillonetto2014kernel}. 
In this section, we review some of their fundamental concepts \cite{cucker2007learning,berlinet2004reproducing} needed for our constructions.}
\subsection{Reproducing Kernel Hilbert Spaces}
A Hilbert space $\CH$ of $f: \BX \to \mathbb{R}$ functions with an inner product $\langle\cdot,\cdot\rangle_{\CH}$ is called a {\em Reproducing Kernel Hilbert Space} (RKHS), if each Dirac functional, which evaluates functions at a point,
$\delta_z: f \to f(z)$, is 
bounded for all $z \in \BX$, that is $\forall z \in \BX: \exists \, \kappa_z > 0$ with $|\hspace{0.3mm}\delta_z(f)\hspace{0.3mm}| \leq \kappa_z\, \| f \|_{\CH}$ for all $f \in \CH$.

Then, by building on the Riesz representation theorem, a unique {\em kernel}, $k: \BX \times \BX \to \mathbb{R}$,  can be constructed
encoding the Dirac functionals satisfying $\langle k(\cdot,z),f \rangle_{\CH} = f(z),$
for all 
$z \in \BX$ and $f \in \CH$, which formula is called the {\em reproducing property}.
As a special case of this property, we also have for all $z,s \in \BX$ {\newpart that} $k(z,s)=\langle k(\cdot,z),k(\cdot,s) \rangle_{\CH}.$ 
Therefore, the kernel of an RKHS is a symmetric and positive-definite function.

Furthermore, the Moore-Aronszajn theorem asserts that the converse statement holds true, as well: 
for every symmetric and positive-definite function $k: \BX \times \BX \to \mathbb{R}$, there exists a unique RKHS for which $k$ is its reproducing kernel \cite{berlinet2004reproducing}.

The {\em Gram} or kernel matrix of a given kernel $k$ w.r.t.\ (input) points $x_1, \dots, x_n$ is 
$K_{i,j} \defeq k(x_i,x_j)$, for all {\newpart  $i, j \in [n] \doteq \{1,\dots, n\}$}. Observe that $K \in \mathbb{R}^{n \times n}$ is always positive semi-definite. A kernel is called {\em strictly} positive-definite, if its Gram matrix is positive-definite for all {\em distinct} inputs $\{x_i\}$.

\subsection{Paley--Wiener Spaces}
A \textit{Paley--Wiener} (PW) space is a subspace of $\mathcal{L}^2 (\mathbb{R}^d)$, where for each $f \in \mathcal{H}$ the {\em support} of the {\em Fourier transform} of $f$ is included in a given hypercube $[-\eta_1, \eta_1\hspace{0.3mm}] \times \dots \times [-\eta_d, \eta_d\hspace{0.3mm}]$, where $\{\eta_k\}$ are positive hyper-parameters.

Paley--Wiener Spaces are RKHSs with the following {\em strictly} positive definite reproducing 
kernel. For all $u,v \in \mathbb{R}^d,$ let
$$k(u,v) \,\defeq\, \frac{1}{\pi^{d}} \prod_{k=1}^d \frac{\sin(\eta_k (u_k-v_k))}{u_k-v_k},$$
where, for convenience, $\sin(\eta \cdot 0)/0$ is defined to be $\eta$ \cite{yang2014quasi}. 

Paley--Wiener spaces are subspaces of $\mathcal{L}^2$; therefore, they inherit its inner product and norm, as well, i.e., $\langle f,g \rangle_\CH = \langle f,g \rangle_{\mathcal{L}^2}$. From now on, we work with the \textit{Paley--Wiener kernel}.

\subsection{Minimum-Norm Interpolation}
\label{sec:min-norm-int}
An important concept for the MiNCE constructions is the {\em minimum-norm interpolant}. Given a dataset of input-output pairs, $\{(x_k,z_k)\}$, where the inputs $\{x_k\}$ are distinct,
the element from $\mathcal{H}$ which interpolates every output $z_k$ at the corresponding input $x_k$, and has the smallest kernel norm, i.e.,
$$\hat{f}_n \,\defeq\, \argmin \big\{\,\|\hspace{0.3mm}f\hspace{0.4mm}\|_{\mathcal{H}} : f \in \mathcal{H}\hspace{1.5mm} \&\hspace{1.5mm} \forall\hspace{0.3mm} k \in [n]: f(x_k) =\, z_k   \,  \big\},$$
exists and it admits the following form \cite{berlinet2004reproducing}:
\begin{equation}
\label{min-norm-interpolant-formula}
\hat{f}_n(x) \,= \,\sum_{k=1}^n \hat{\alpha}_k k(x,x_k),
\end{equation}
for all 
$x \in \mathbb{R}^d$, where the weights are $\hat{\alpha} = K^{-1} z$ with $z \defeq (z_1,...,z_n)\tr$ and $\hat{\alpha} \defeq (\hat{\alpha}_1,...,\hat{\alpha}_n)\tr$, assuming $K$ is invertible.
From the reproducing property, we obtain $\|\hat{f}_n\|_\mathcal{H}^2 = \hat{\alpha}\tr K \hat{\alpha} = z\tr K^{-1}z.$ The kernel norm can be seen as a {\em measure of smoothness}. In  Paley--Wiener spaces it coincides with the $\mathcal{L}^2$ norm, which can measure the {\em energy} of signals, as well.

\section{Nonparametric {\newpart Confidence Bands}}

\subsection{Data Generation}
Let us consider $\CZ_n \defeq \{(x_1,y_1), (x_2,y_2), ..., (x_n,y_n)\}$, a finite sample of i.i.d.\ input-output data with an {\em unknown} $\mathbb{P}_{X,Y}$ joint distribution, where $x_k \in \mathbb{R}^d$ is an input vector and $y_k \in \mathbb{R}$ is a scalar output with $\mathbb{E}\big[y_k^2\big] < \infty$.
We will assume that
$$y_k \,=\, f_*(x_k)+\varepsilon_k,$$
 with $\mathbb{E}\big[\varepsilon_k\big] = 0,$ for $k \in [n] \defeq \{1,...,n\},$ where $\{\varepsilon_k\}$ are the {\em  noises} on the \textit{regression function} $f_*(t) = \mathbb{E} [\hspace{0.5mm}Y\hspace{0.4mm} |\hspace{0.5mm} X = t\hspace{0.6mm}].$

\subsection{Objectives: Coverage Guarantees and Strong Consistency}

Our primary goal is to construct {\em nonparametric} confidence bands for the unknown $f_*$ data-generating function that have guaranteed (user-chosen) coverage probabilities for {\em finite}, and possibly small number of data points. Moreover, the construction should be {\em distribution-free} w.r.t.\ the measurement noises, and it should be {\em simultaneously} guaranteed for all inputs.

More precisely, based on dataset $\CZ_n$, we build a {\em confidence band} $B_n: \mathbb{R}^d \to \mathbb{R} \times \mathbb{R}$, where $B_n(x) = (L_n(x), U_n(x))$ specifies the {\em endpoints} of an {\em interval estimate} for $f_*(x)$ at input $x \in \mathbb{R}^d$, such that for a given $\alpha \in (0,1)$ {\em risk probability},
    \begin{equation*}
    \nu(B_n) \,\defeq\, \mathbb{P}\big(\hspace{0.3mm}\forall x \in \mathbb{R}^d: L_n(x) \leq f_*(x) \leq U_n(x)\hspace{0.3mm}\big)\, \geq\, 1-\alpha,
\end{equation*}
    where $\nu(B_n)$ is called the \textit{reliability} of the confidence band.

We also want our confidence bands to {\em shrink} as the sample size increases. A confidence band construction is called {\em strongly uniformly consistent}, if as $n \to \infty$, we have
\begin{equation*}
\max\! \Big\{\sup_{x \in \RR^d} |\hspace{0.1mm}L_n(x) - f_*(x)|, \sup_{x \in \RR^d} |\hspace{0.3mm}U_n(x) - f_*(x)|\Big\} \convergealmostsurely\, 0.
\end{equation*}

This does not only mean that, asymptotically, $\forall x :L_{\infty}(x)=U_{\infty}(x) = f_* (x)$, but also that the convergence is {\em uniform} in all inputs. 
Therefore, the width of the confidence band should (a.s.) converge to zero uniformly over the domain of inputs.

\subsection{Main Assumptions}
The principal assumptions on the data are as follows:

\smallskip
\begin{assumption}
\label{assumption-iid-multivariate} 
{\em The dataset $(x_1, y_1), \dots, (x_n, y_n) \in \mathbb{R}^d \times \mathbb{R}$ is an i.i.d.\ sample of inputs and outputs; and $\{y_k\}$ have finite variance.}
\end{assumption}
\smallskip

\begin{assumption}
\label{assumption-positive-input-density} 
{\em The distribution of the inputs is known, it is absolutely continuous, and its density $h_*(x) > 0$, for all $x\in \mathbb{R}^d$.}
\end{assumption}
\smallskip

\begin{assumption}
\label{assumption-Paley-Wiener-space-multivariate} 
{\em Function $f_*$ is from a Paley--Wiener space and there is a (universal) constant $\varrho > 0$ with $\forall\,x \in \mathbb{R}^d:  
f_*^2(x) \leq \varrho \, h_*(x).$ }
\end{assumption}
\smallskip

The i.i.d.\ requirement in A\ref{assumption-iid-multivariate} is standard in statistics and related fields, such as machine learning. The assumption on the full support of the input distribution, A\ref{assumption-positive-input-density}, is needed to eventually obtain information about the {\em whole} regression function. For simplicity, A\ref{assumption-positive-input-density} also assumes that the input distribution is known. This simplifies the theoretical analysis, and it is a relatively mild requirement in practice: the sampling mechanism is often known, and even when it is not, it can typically be estimated and used as a plug-in for $h_*$. We leave the analysis of estimated input distributions to future work.

The last assumption, A\ref{assumption-Paley-Wiener-space-multivariate}, serves two purposes: 
it restricts the model class to Paley--Wiener spaces, and it ensures that random observations 
carry enough information to be able to successfully {\em generalize} to unobserved inputs. The constant $\varrho$ measures  how ``informative'' querying $f_*$ at a random input is: a smaller $\varrho$ means that the 
observations are concentrated on the more ``interesting'' (higher energy) parts of function $f_*$.

This assumption also places an implicit restriction on the {\em tail decay rate} of the input 
distribution $h_*$: since $f_*^2(x) \leq \varrho\, h_*(x)$ must hold everywhere,
the tails of $h_*$ cannot decay faster than $f_*^2$. 
By the Paley--Wiener theorem \cite{rudin1987real}, non-zero band-limited functions extend to entire functions of exponential type and hence, by the 
log-integral condition for the Cartwright class, they cannot decay super-exponentially.
Thus, 
distributions with super-exponentially decaying tails, such as the Gaussian, are incompatible with this assumption. In our experiments, we 
used input 
distributions with polynomially decaying tails, such as the Student-$t$ distribution, 
which are compatible with A\ref{assumption-Paley-Wiener-space-multivariate}.

\subsection{Overview of the Construction}

In this section, we briefly overview the main ideas behind the MiNCE construction, also see \cite{csaji2022nonparametric}. The main steps are:
\smallskip

\begin{enumerate}
    \item[(i)] First, we build a guaranteed {\em simultaneous confidence set} for some of true (noiseless) outputs of the target function at the {\em observed} inputs. Namely, we construct a set $\Theta \subseteq \mathbb{R}^{n_0}$ that stochastically guarantees to contain $f_*(x_k)$, for $k \in [n_0]$, where $n_0 \leq n$ is user-chosen (as the data is i.i.d., it is w.l.o.g.\ that we use the first $n_0$). 
    This is a nontrivial task, nonetheless it is ``easier'' than constructing a confidence band for the  {\em whole} function. For this step, e.g., we can build on the results of \cite{csaji2019distribution}.
    \smallskip
    \item[(ii)] Using the confidence set $\Theta \subseteq \mathbb{R}^{n_0}$ for the true values of $f_*$ at some {\em observed} inputs, constructed in step (i), we calculate a high probability {\em upper bound}, $\tau$, for the kernel {\em norm} (square) of the true function (recall that this norm is a smoothness measure of the target function).
    \smallskip
    \item[(iii)] Then, for each {\em input query} point $x_0 \in \mathbb{R}^d$, we can construct a {\em confidence interval} for $f_*(x_0)$ as follows. We keep a candidate value $z_0$ in the confidence region {\em if and only if} there is a $z = (z_1, \dots, z_{n_0})\tr \in \Theta$, such that the {\em minimum-norm interpolation} of the dataset $\{(x_k,z_k)\}_{k=1}^{n_0} \cup \{(x_0,z_0) \}$ has a norm (square) less than or equal to $\tau$ (i.e., our upper bound for $\norm{f_*}_\mathcal{H}^2$).
\end{enumerate}
\smallskip

In order to make this approach practical, apart from the method in (i), we need a way to guarantee an upper bound for the norm (square) of the true function. Additionally, we also need to give an efficient method to compute the endpoints of the confidence intervals for any query input $x_0 \in \mathbb{R}^d$.

Note that step (i) is superfluous in those cases, when there are no measurement noises, that is if $\varepsilon_k = 0$ for $k \in [n]$, since then we can directly work with the $\Theta = \{f_*(x_k)\}$ values.

\section{Outputs without Measurement Noises}
\label{sec-noise-free-known density}

We start with a simplified problem. In this case, we assume that we perfectly observe the regression function at random inputs (no measurement noises), that is, $\forall\,k \in [n]$, $y_k = f_*(x_k)$.

In this noise-free setup, we can recall the Nyquist–Shannon sampling theorem, according to which a band-limited function can be perfectly reconstructed from the (equidistant) samples, assuming that the sampling rate exceeds twice the maximum frequency. However, if we just have a few random observations, we would still like to have at least a region estimate.

\subsection{Norm Bounds and Abstract Regions}
To estimate $\|f_*\|^2_{\CH}$, MiNCE applies {\em importance sampling} \cite{tokdar2010importance}. 
By the {\em strong law of large numbers} (SLLN), we have
\begin{equation}
\label{noise-free-density-convergence}
\begin{aligned}
\frac{1}{n} \sum_{i=1}^n \frac{f_*^2(x_i)}{h_*(x_i)}\,  \convergealmostsurely &\;\,\mathbb{E} \biggr[ \frac{f_*^2(x)}{h_*(x)} \biggr]  = \int_{\mathbb{R}^d} \frac{f_*^2(t)}{h_*(t)} h_*(t) \mbox{d}t \\ &= \int_{\mathbb{R}^d} f_*^2(t) \mbox{d}t = \norm{f_*}_2^2 = \norm{f_*}_\CH^2.
\end{aligned}
\end{equation}
Based on this idea, we can construct a guaranteed upper bound for the norm (square) of the regression function \cite{horvath2023nonparametric} by 

\medskip
\begin{lemma}
\label{lemma-norm-estimation-known-density-noisefree}
{\em
Assume A\ref{assumption-iid-multivariate}, A\ref{assumption-positive-input-density}, A\ref{assumption-Paley-Wiener-space-multivariate} and that $y_k = f_*(x_k)$ for $k \in [n]$. 
Then, for any $\alpha \in (0,1)$ and sample size $n$, we have
$$\mathbb{P}( \norm{f_*}_\mathcal{H}^2 \leq \kappa_n) \,\geq \,1-\alpha,$$
with the following choice of the upper bound $\kappa_n:$
$$\kappa_n\, \defeq\, \frac{1}{n} \sum_{k=1}^n \frac{y_k^2}{h_*(x_k)} + \varrho \, \sqrt{\frac{\ln(1/\alpha)}{2n}}.$$}
\end{lemma}
\smallskip
Lemma \ref{lemma-norm-estimation-known-density-noisefree} was originally given in \cite{horvath2023nonparametric} without a proof. Now, for completeness, we include its proof in Appendix \ref{appendix-lemma1-proof}. It builds on Hoeffding's inequality, using that $\forall\, x: f_*^2(x)/h_*(x) \leq \varrho.$

\smallskip
\begin{remark}
{\em The upper bound of Lemma \ref{lemma-norm-estimation-known-density-noisefree} can be
replaced by any concentration inequality providing the same guarantee. In \cite{11052262}, we studied using
a uniformly randomized Hoeffding and an empirical Bernstein inequality. In the
latter case, 
\begin{equation}
\label{eq:emp-bernstein}
\kappa'_n \,\defeq\, \frac{1}{n} \sum_{k=1}^n \frac{y_k^2}{h_*(x_k)}
+ \sqrt{\frac{2\hspace{0.3mm}V_n \ln (2/\alpha)}{n}}
+ \varrho\,\frac{7 \ln (2/\alpha)}{3(n-1)},
\end{equation}
 where $V_n$ is the unbiased empirical variance of 
$\{\hspace{0.3mm}y_k^2/h_*(x_k)\hspace{0.3mm}\}$.
}
\end{remark}
\smallskip

Recall that the dataset is $\CZ_n \defeq \{ (x_1,y_1), \dots, (x_n,y_n) \}$.
In the {\em abstract} sense, our {\em confidence region} in the  RKHS is
\begin{align}
    \label{eq:absreg}
    \mathcal{C}_n\, \defeq\, \big\{\hspace{0.3mm} f \in \CH \,\mid\; &  f\hspace{-0.3mm}
    \interpolates\hspace{-0.3mm}\CZ_n \,\land\,  \norm{f}_\mathcal{H}^2 \leq \kappa_n \hspace{0.3mm}\big\},
\end{align}
where $\kappa_n$ is a norm bound with $\mathbb{P}\big(\norm{f_*}_{\CH}^2 \leq \kappa_n \hspace{0.3mm}\big) \geq 1-\alpha$ and 
``$f \interpolates \CZ_n$'' denotes that $f$ {\em interpolates} $\{(x_i, y_i)\}$, that is,
$$f \interpolates \CZ_n\, \stackrel{\text{def}}{\iff}\,
\forall\, i \in [n]: f(x_i) =\hspace{0.3mm} y_i.$$
 
As $f_* \interpolates \CZ_n$, since now $\forall i: y_i = f_*(x_i)$, we clearly have\, $\mathbb{P}\big(f_*\in \mathcal{C}_n \hspace{0.3mm}\big) \geq 1-\alpha$,\, but $\mathcal{C}_n$ is hard to construct in practice. We need an efficient algorithm to build the confidence band.

\subsection{Confidence Bands and Interval Endpoints}
We can construct a {\em confidence band} $\CB_n$ from the {\em abstract} region $\mathcal{C}_n$, by computing an envelope for the potential outputs,
\begin{equation}
\label{eq:confband}
\CB_n \defeq \big\{ f \in \CH \mid \forall x\in \mathbb{R}^d: L_n(x) \leq f(x) \leq U_n(x) \big\},
\end{equation}
where the {\em interval endpoints} for input $x \in \mathbb{R}^d$ are defined 
by
\begin{equation}
\label{eq:confband2}
L_n(x) \defeq \inf_{f \in \mathcal{C}_n}\! f(x),\qquad \text{and}\qquad U_n(x) \defeq \sup_{f \in \mathcal{C}_n}\! f(x).
\end{equation}
From the construction, it is immediate that we have $\mathcal{C}_n \subseteq \CB_n$. Hence, 
$\norm{f_*}_\mathcal{H}^2 \leq \kappa_n \Rightarrow f_* \in \mathcal{C}_n \Rightarrow  f_* \in \CB_n$. Combining this argument with Lemma \ref{lemma-norm-estimation-known-density-noisefree}, we get the following guarantee \cite{horvath2023nonparametric}

\smallskip
\begin{theorem}
\label{theorem-reliability-noisefree}
{\em Assume A\ref{assumption-iid-multivariate}, A\ref{assumption-positive-input-density}, A\ref{assumption-Paley-Wiener-space-multivariate} and that $y_k = f_*(x_k)$ for $k \in [n]$. Then, for any risk probability $\alpha \in (0,1)$ and sample size $n \in \mathbb{N}$, the {\newpart confidence} band has coverage probability
$$\nu(B_n)\, =\, \mathbb{P}(f_* \in \CB_n)\, \geq\, 1-\alpha.$$}
\end{theorem}

Note the distinction between $B_n$ and $\mathcal{B}_n$: the former is a mapping that defines confidence intervals for each input, while the latter is the functional subset of the Hilbert space containing all functions consistent with those bounds.

Moreover, for any given {\em query input} $x_0 \in \mathbb{R}^d$, the endpoints of the confidence interval for $f_*(x_0)$, see \eqref{eq:confband2}, can be constructed {\em analytically} \cite{csaji2022nonparametric}. We can exploit that the {\em minimum-norm interpolant} \eqref{min-norm-interpolant-formula} of an input-output dataset, and hence the minimally needed kernel norm, can be calculated analytically.

If $x_0 = x_k$ for any $k \in [n]$, the confidence interval at $x_0$ is $L_n(x_0)=U_n(x_0)=y_k$.  
Henceforth assume that $x_0 \notin \{x_k\}$. First, the Gram matrix is extended 
with the query input $x_0$, 
    $$K_0(i+1,j+1)\, \defeq\, k(x_i,x_j),$$
    for $i, j = 0,1,...,n$.
Then, we compute the highest and the lowest $y_0$ values which can be interpolated with a function from $\CH$ having {\em at most} $\kappa_n$ as its squared norm. This task leads to the following two {\em convex} optimization problems \cite{csaji2022nonparametric}:
    \begin{equation}
    \label{noiseless-opt-min-max}
    \begin{split}
    \mbox{min\,/\,max} &\quad y_{0} \\[0.5mm]
    \mbox{subject to} &\quad (y_0, y\tr)  \gramnplus^{-1} (y_0, y\tr)\tr \leq\, \kappa_n\\[1mm]
    \end{split}
    \end{equation}
    where ``min\,/\,max'' means that we have to separately solve the problem as a minimization and also as a maximization.

    Note that $K_0$ is (a.s.) invertible, as the Paley--Wiener kernel is {\em strictly} positive definite and $x_0, x_1, \dots, x_n$ are (a.s.) distinct.
    
Both problems of \eqref{noiseless-opt-min-max} can be solved {\em analytically} using the technique of {\em Lagrange multipliers} \cite{csaji2022nonparametric}. In order to obtain their solutions, let us partition the inverse of matrix $K_0$, as
$$
 \begin{bmatrix}
 \; c & b\tr\\
 \; b & A
 \,\end{bmatrix} \defeq\, \gramnplus^{-1}.
$$
The solutions of the min and max problems \eqref{noiseless-opt-min-max} are exactly the solutions of
$a_0 y_0^2 + b_0 y_0 + c_0 = 0$, where  $a_0 \defeq c$, $b_0 \defeq 2b\tr y$ and $c_0 = y\tr\hspace{-0.3mm} A y - \kappa_n$. If there is no solution, $B_n(x_0) \defeq \emptyset$; otherwise $B_n(x_0) \defeq (\hspace{0.3mm}y_{\mathrm{min}},\, y_{\mathrm{max}}\hspace{0.3mm})$, where $y_{\mathrm{min}} \leq y_{\mathrm{max}}$ are the solutions of the quadratic equation (they can coincide).

By exploiting {\em Schur complements} \cite{horvath2025single}, we do not need to invert $K_0$ for every $x_0$. Combining this with the quadratic formula, the confidence interval at $x_0$ can also be written as
\begin{equation*}
B_n(x_0) = \hat{f}_n(x_0) \pm\! \sqrt{(k(x_0,x_0)-k_0\tr K^{-1}k_0)\,(\kappa_n-\hat{\kappa}_n)},
\end{equation*}
if $\kappa_n \geq \hat{\kappa}_n$, and $B_n(x_0) = \emptyset$ otherwise, where $\hat{f}_n \in \CH$ is the {\em minimum-norm interpolant} of dataset $\CZ_n$, $\hat{\kappa}_n \doteq \|\hat{f}_n\|_\CH^2$, 
$k_0 = (k(x_0,x_1), \dots, k(x_0, x_n))\tr$, and $K$ is the kernel matrix based on $x_1, \dots, x_n$. Thus, the bands are centered around $\hat{f}_n$.

\begin{figure}[!t]
    \centering
	\hspace*{-2mm}	 	
	\includegraphics[width = \columnwidth]{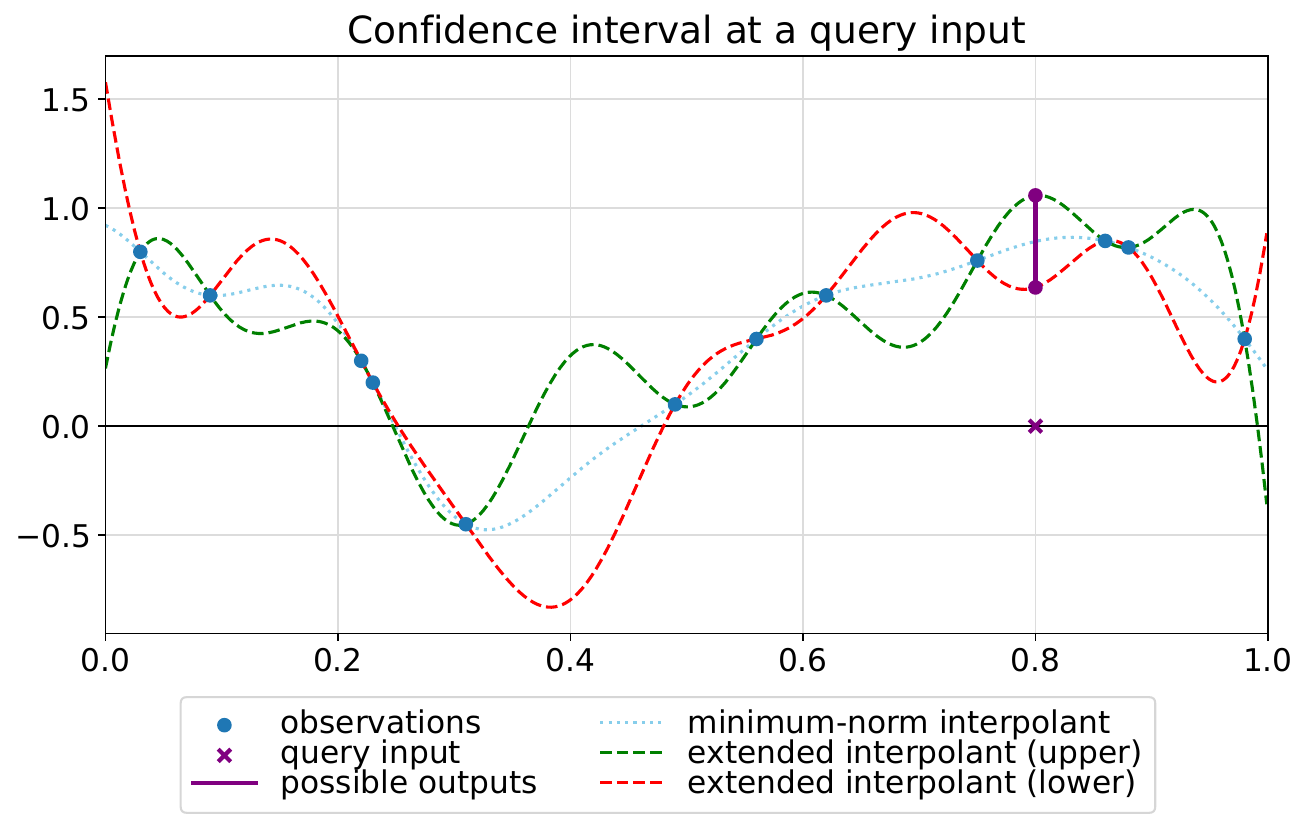} 	
    \caption{The endpoints of the confidence interval at a query input illustrating the construction of the confidence band $\CB_n$ according to \eqref{eq:absreg}, \eqref{eq:confband} and \eqref{eq:confband2}.}
\label{fig:interpolants}
\end{figure}

Figure \ref{fig:interpolants} illustrates the core idea of the method. For a fixed query input $x_0$, the {\em minimum} and the {\em maximum} $y_0$ values satisfying \eqref{noiseless-opt-min-max} specify the {\em endpoints} of the confidence interval.

\subsection{Strong Uniform Consistency}
Now, we will show that the MiNCE construction not only guarantees the inclusion of the regression function with a user-chosen (nonasymptotic) probability, but also that the resulting confidence bands are {\em strongly uniformly consistent}. 

Our first observation is about our kernel norm bounds:
\smallskip

\begin{lemma}
\label{norm-consistency}
{\em
Assuming A\ref{assumption-iid-multivariate}, A\ref{assumption-positive-input-density}, A\ref{assumption-Paley-Wiener-space-multivariate} and $y_k = f_*(x_k)$ for $k \in [n]$, then the norm bounds are consistent, that is,  as $n \to \infty$, 
$$\kappa_n \defeq \underbrace{\frac{1}{n}\sum_{k=1}^n \frac{y_k^2}{h_*(x_k)}}_{\kappa_{n,1}} + \underbrace{\varrho\,\sqrt{\frac{\ln (1/\alpha)}{2n}}}_{\kappa_{n,2}} \convergealmostsurely  \kappa_*\defeq \norm{f_*}_{\CH}^2.\vspace{-2mm}$$}
\end{lemma}
\smallskip

\begin{proof}
Clearly, we have $\kappa_{n,2} \to 0$. Moreover, by the strong law of large numbers, $\kappa_{n,1} \convergealmostsurely \norm{f_*}_{\CH}^2$, as we saw in \eqref{noise-free-density-convergence}.
\end{proof}

Next, we argue that it is sufficient to prove the uniform consistency of the {\em abstract} sets $\{\mathcal{C}_n\}$ in the RKHS norm, as this implies the consistency of the {\em confidence bands}, $\{\CB_n\}$.
\smallskip
\begin{theorem}
\label{theorem-consistency-abstract-endpoint-based}
{\em Let $\CH$ be any RKHS of\hspace{0.3mm} $\mathbb{R}^d \to \mathbb{R}$ functions with a bounded kernel $k$. 
Let $f_* \in \CH$ and let $\{\CC_n\}$ be a sequence of nonempty confidence regions in $\CH$, such that, as $n \to \infty$,
$$
\sup_{f\in\mathcal{C}_n}\!\|f - f_*\|_{\mathcal{H}}\,\convergealmostsurely\;0.
$$
Let $\{\CB_n\}$ be a confidence band sequence defined by \eqref{eq:confband} and \eqref{eq:confband2} based on regions $\{\CC_n\}$. Then, as $n \to \infty$, we have 
\begin{equation*}
\max\hspace{-0.5mm}\big\{\hspace{0.5mm} \|L_n - f_*\|_\infty, \|\hspace{0.3mm} U_n - f_*\|_\infty \big\} \convergealmostsurely\, 0.
\end{equation*}}
\end{theorem}
The proof of Theorem \ref{theorem-consistency-abstract-endpoint-based} can be found in Appendix \ref{appendix-abst-to-band-consistency-proof}.

Therefore, it is enough to show the consistency of $\{\CC_n\}$. In fact, we will show the consistency of a sequence of even larger sets. Let $\kappa^*_n \defeq \max \{ \kappa_*, \kappa_n \}$, where $\kappa_* \defeq \norm{f_*}_\CH^2 $. Then
\begin{equation}
\label{eq:absstarreg}
\mathcal{C}^*_n \defeq \big\{\hspace{0.3mm} f \in \CH: f \interpolates \CZ_n\,\land\, \norm{f}_\CH^2 \leq \kappa^*_n\hspace{0.3mm}\big\}.
\end{equation}
Obviously, $\mathcal{C}_n \subseteq \mathcal{C}^*_n$, so if we show the consistency of $\{\mathcal{C}^*_n\}$ it also implies the consistency of $\{\mathcal{C}_n\}$, which is sufficient to show the consistency of our confidence bands $\{\CB_n\}$. Of course, $\mathcal{C}^*_n$ is just a theoretical construction, since we do not know $\kappa_*$ in practice. It will only serve as a theoretical tool.

Let $\CH_n \defeq \text{span}\{\hspace{0.3mm} k(\cdot,x_k): k \in [n] \hspace{0.3mm}\}$. Recall that $\hat{f}_n \in \CH$ is the {\em minimum-norm interpolant} of dataset $\CZ_n$. Then, $\hat{f}_n \in \CH_n$, as discussed in Section \ref{sec:min-norm-int}. As for all $f \in \CH$, $f \interpolates \CZ_n$ implies $\|\hat{f}_n\|_\CH \leq \|f\|_\CH$, and $f_* \interpolates \CZ_n$ (noise-free observations), we have $\|\hat{f}_n\|^2_\CH \leq \kappa_* = \|f_*\|^2_\CH$. Hence, we have $f_*, \hat{f}_n \in \CC^*_n$.

\smallskip
\begin{remark}
\label{remark-min-interp-norm-conv}
{\em An important property of $\hat{f}_n$
is that, in the noise-free case, it is the orthogonal projection of $f_*$ onto $\CH_n$, as the interpolants of $\CZ_n$ form the affine set $f_* + \CH_n^{\bot}$. Under A\ref{assumption-iid-multivariate} and A\ref{assumption-positive-input-density}, the inputs are (a.s.) dense in $\mathbb{R}^d$, hence $g \perp \CH_n$ for all $n$ implies $g \equiv 0$, as every $g \in \CH$ is continuous, because the Paley--Wiener kernel is continuous; therefore, $\overline{\bigcup_n \CH_n} = \CH$ and, using A\ref{assumption-Paley-Wiener-space-multivariate} ($f_* \in \CH$), $\|\hat{f}_n-f_*\|_\CH \convergealmostsurely 0$, as $n\to \infty$, and
\begin{equation*}
    \hat{\kappa}_n \doteq \|\hat{f}_{n}\|_{\mathcal{H}}^2 = z_*\tr K_n^{-1} z_*\convergealmostsurely \|f_*\|_{\mathcal{H}}^2 = \kappa_*, \qquad\text{as}\;\, n \to \infty,
\end{equation*}
by the continuous mapping theorem, as $\|\cdot\|_{\mathcal{H}}$ is continuous, where $K_n$ is the kernel matrix for the first $n$ inputs and $z_*$ is the vector of noiseless outputs, $z_* \doteq (f_*(x_1), \dots, f_*(x_n))\tr$.}
\end{remark}
\smallskip

As $\mathcal{H}_n$ is a {\em closed} subspace of $\mathcal{H}$, since it is finite dimensional,  from the Hilbert projection theorem \cite{rudin1987real}, we know that any $f \in \CH$ can be {\em uniquely} decomposed as
$$f = f_{\parallel } + f_{\bot},$$
where $f_{\parallel} \in \mathcal{H}_n$ and $f_{\bot} \in \mathcal{H}_n^{\bot}$, the orthogonal complement of $\CH_n$. Interestingly, for $f' \in \mathcal{C}^*_n$, $f'_{\parallel}$ is the min-norm interpolant:\!\!\!\!
\smallskip
\begin{lemma}
\label{parallel-perpendicular-lemma}
{\em Assuming A\ref{assumption-iid-multivariate}, A\ref{assumption-positive-input-density}, A\ref{assumption-Paley-Wiener-space-multivariate} and that $y_k = f_*(x_k)$ holds for $k \in [n]$, for every $f' \in \mathcal{C}^*_n$, it holds true that $f'_{\parallel} = \hat{f}_n$.}
\end{lemma}
\smallskip
The proof of Lemma \ref{parallel-perpendicular-lemma} can be found in Appendix \ref{appendix-parallel-perpendicular-lemma}.

By recalling that $\hat{\kappa}_n \defeq \|\hat{f}_n\|_{\CH}^2,$ from Lemma \ref{parallel-perpendicular-lemma}, we get
\begin{equation*}
\begin{aligned}
\hat{\kappa}_n &\,=\, \|\hat{f}_n\|^2_\CH \leq\, \|f'\|^2_{\CH} \,=\, \|f'_{\parallel} + f'_{\bot}\|^2_{\CH} \\
&\,=\, \|f'_{\parallel }\|^2_{\CH} + \, \|f'_{\bot}\|^2_{\CH} \,= \,\|\hat{f}_n\|^2_{\CH} +\, \|f'_{\bot}\|^2_{\CH} \,\leq\, \kappa^*_n,
\end{aligned}
\end{equation*}
for all $f' \in \CC^*_n$; from which, after rearrangement, we get
$$\|f'_{\bot}\|_{\CH} \leq \sqrt{\kappa^*_n - \hat{\kappa}_n}.$$
Using this, we can formulate a key lemma for consistency:

\smallskip
\begin{lemma}
\label{consistency-noise-free-known-density-2}
{\em Assuming A\ref{assumption-iid-multivariate}, A\ref{assumption-positive-input-density}, A\ref{assumption-Paley-Wiener-space-multivariate} and that $y_k = f_*(x_k)$ holds for $k \in [n]$. Then, for any sample size $n$, we have
$$\sup_{f \in \mathcal{C}^*_n}\!\|\hat{f}_n-f\|_\CH \, \leq \,\sqrt{\kappa^*_n-\hat{\kappa}_n},\vspace{-1mm}$$
where $\kappa^*_n \defeq \max \{ \kappa_*, \kappa_n \}$ and $\hat{\kappa}_n \defeq \|\hat{f}_n\|_{\CH}^2.$}
\end{lemma}

\smallskip

\begin{proof}
For an arbitrary $f \in \mathcal{C}^*_n$, using Lemma \ref{parallel-perpendicular-lemma}, we have
\begin{equation*}
\begin{aligned}
\|\hat{f}_n-f\|_{\CH}&= \|\hat{f}_n - (f_{\parallel } + f_{\bot})\|_{\CH}\\
&=\|\hat{f}_n  - \hat{f}_n - f_{\bot}\|_{\CH}\\
&=  \norm{f_{\bot}}_{\CH} \leq \sqrt{\kappa^*_n - \hat{\kappa}_n}.\\[-4mm]
\end{aligned}
\end{equation*}
\end{proof}
Now, we are ready to state one of our main contributions: 
\smallskip
\begin{theorem}
\label{consistency-noisefree-known-density}
{\em Assuming A\ref{assumption-iid-multivariate}, A\ref{assumption-positive-input-density}, A\ref{assumption-Paley-Wiener-space-multivariate} and that $y_k = f_*(x_k)$ holds for $k \in [n]$. Then, as $n \to \infty$, we have
$$
\sup_{f\in\mathcal{C}_n}\!\|f - f_*\|_{\mathcal{H}}\,\convergealmostsurely\;0.
$$
}
\end{theorem}
\begin{proof}
From the constructions \eqref{eq:absreg}, \eqref{eq:absstarreg} and Lemma \ref{consistency-noise-free-known-density-2}, we have $\CC_n \subseteq \CC^*_n \subseteq \bar{B}(f_*,2\sqrt{\kappa^*_n-\hat{\kappa}_n})$, where $\bar{B}(f,r)$ denotes a closed norm ball in $\CH$ with center $f \in \CH$ and radius $r$. The factor $2$ comes from the triangle inequality (through $\hat{f}_n$), as $f_* \in \CC^*_n$, too. We also have $2\sqrt{\kappa^*_n-\hat{\kappa}_n} \convergealmostsurely 0$, 
from Lemma \ref{norm-consistency} and 
the fact that $\hat{\kappa}_n \convergealmostsurely \kappa_*$,
as discussed in Remark \ref{remark-min-interp-norm-conv}.
\end{proof}
\vspace{-2mm}

Combining Theorems \ref{theorem-consistency-abstract-endpoint-based} and \ref{consistency-noisefree-known-density}, we can conclude:
\smallskip
\begin{corollary}
{\em Under the assumptions of Theorem \ref{consistency-noisefree-known-density}, we have
\begin{equation*}
\max\hspace{-0.5mm}\big\{\hspace{0.5mm} \|L_n - f_*\|_\infty, \|\hspace{0.3mm} U_n - f_*\|_\infty \big\} \convergealmostsurely\, 0, \quad\text{as}\quad n\to\infty.
\end{equation*}}
\end{corollary}
\smallskip
Hence, the MiNCE bands are {\em strongly uniformly consistent.}

Finally, we note that Lemmas \ref {lemma-norm-estimation-known-density-noisefree} and \ref{consistency-noise-free-known-density-2} can be exploited to build a norm-ball {\em outer approximation} of the region, that is
\begin{align}
\label{eq:outer-ball}
\CC_n \subseteq \mathcal{O}_n &\doteq \big\{\, f \in \CH : \| \hat{f}_n - f\|_\CH \leq r_n \hspace{0.3mm}\big\},
\end{align}
with radius $r_n \doteq \sqrt{|\kappa_n - \hat{\kappa}_n|_+}$, where $\kappa_n$ is the stochastic bound for $\|f_*\|_\CH^2$ constructed in Lemma \ref{lemma-norm-estimation-known-density-noisefree} and $\hat{\kappa}_n = \|\hat{f}_n\|_\CH^2$.

\smallskip
\begin{corollary}
\label{thn:outer-ball}
{\em Under the assumptions of Theorem \ref{consistency-noisefree-known-density}, we have
\begin{equation*}
\forall n:\mathbb{P}(f_* \in \mathcal{O}_n) \geq 1-\alpha,\quad\text{and}\quad
\sup_{f\in\mathcal{O}_n}\!\|f - f_*\|_{\mathcal{H}}\econvergealmostsurely\,0.\vspace{-2mm}
\end{equation*}}
\end{corollary}
\smallskip
\begin{proof}
The first part follows form Lemma \ref{lemma-norm-estimation-known-density-noisefree} and
\eqref{eq:outer-ball}, and the second from the same argument that we used in the proof of Theorem \ref{consistency-noisefree-known-density}, i.e., $2\sqrt{\kappa^*_n-\hat{\kappa}_n} \convergealmostsurely 0 \implies r_n \convergealmostsurely 0,$ as $\kappa_n \convergealmostsurely \kappa^*$, by Lemma \ref{norm-consistency}, and thus, $\kappa^*_n\convergealmostsurely \kappa_*$, as well.
\end{proof}

\section{Outputs with Measurement Noise}
\label{noisy-outputs-known-density-section}
Now, we relax the assumption that the regression function is perfectly observed at the sample inputs, as the {\em outputs} are {\em corrupted by noise}, that is, $y_k \,=\, f_*(x_k)+\varepsilon_k$. Nevertheless, we will still assume that the distribution of the inputs is a priori known. Our main aim is to show that MiNCE remains {\em strongly uniformly consistent} under these conditions, as well.

\subsection{Norm Estimation from Noisy Outputs}
The first challenge of the noisy case is that the norm-bound construction of Lemma \ref{lemma-norm-estimation-known-density-noisefree} is no longer directly applicable, as the values $\{f_*(x_k)\}$ are not observed. In order to overcome this, we
will apply a method that can construct a confidence ellipsoid for the true outputs at a subset of the {\em sample} inputs. 
Note that this is a  simpler problem than building a confidence band for the entire function; and there are already algorithms for this. In our experiments, we applied the {\em Kernel Gradient Perturbation} (KGP) method \cite{csaji2019distribution, csaji2023improving}, but any method can be used which can guarantee the following assumption:
\smallskip
\begin{assumption}
\label{assumption-ellipsoid-guarantee}
{\em There is $n_0 = n_0(n) \leq n$, such that for all $\beta \in (0,1)$,  
we can build an ellipsoid $\ConfEll \subseteq \mathbb{R}^{n_0}$, based on $\CZ_n$, with}
\begin{equation*}
\mathbb{P}\big(\, (f_*(x_1), \dots, f_*(x_{n_0}))\tr\hspace{-0.5mm} \in \ConfEll\,\big) \, \geq \, 1-\beta.
\end{equation*}
\end{assumption}
\smallskip
Let us denote the center vector and the shape matrix of this confidence ellipsoid by $\hat{z}_n \in \RR^{n_0}$ and $P_n \in \RR^{n_0 \times n_0}$, that is
\begin{equation}
\label{eq:conf-ell-set}
\ConfEll = \big\{ z \in \RR^{n_0}: (z-\hat{z}_n)\tr P_n\hspace{0.3mm}(z-\hat{z}_n) \leq 1\big\},
\end{equation}
where, w.l.o.g., the radius is assumed to be $1$ (rescale $P$).

\smallskip
\begin{remark}
{\em A set of possible assumptions under which such ellipsoids can be constructed, used by the KGP method, is: the noises $\{\varepsilon_k\}$ have zero mean, they are independent of the inputs, $\{x_k\}$, moreover, 
$\varepsilon \defeq (\varepsilon_1, \dots, \varepsilon_n)\tr$ is distributionally invariant 
with respect to a known compact\vspace{-0.6mm} matrix group $\mathcal{G}$. 

This concept is defined as: an $\mathbb{R}^n$-valued random vector $\varepsilon$ is distributionally invariant w.r.t.\ a compact group of transformations, $(\CG, \circ)$, where ``$\circ$'' denotes the function composition and each $G \in \CG$ maps $\RR^n$ to itself, if for all $G \in \CG$, random vectors $\varepsilon$ and $G(\varepsilon)$ have the same (joint) distribution.

Two typical examples having this property are the following: if $\{ \varepsilon_i \}$ are {exchangeable}, then we can use the finite group of permutations on the noise vector. If $\{ \varepsilon_i \}$ are independent and {symmetric} about zero, then we can apply the finite group of diagonal matrices with $\pm 1$ entries, i.e., sign-changes \cite{csaji2019distribution}.

}
\end{remark}
\smallskip

Given such a confidence ellipsoid, we can extend the results of Lemma \ref{lemma-norm-estimation-known-density-noisefree} to the noisy case by computing the worst-case bound over all output vectors within the ellipsoid. By using the notation $z \defeq (z_1, \dots, z_{n_0})\tr,$ the resulting problem is:
\begin{equation}
\label{noisy-norm-max-rewritten}
\begin{split}
\mbox{maximize} &\quad \frac{1}{n_0}\, \sum_{k=1}^{n_0} \left(\frac{{z_k}^2}{h_*(x_k)} \land \varrho\right) + \varrho \, \sqrt{\frac{\ln (1/\alpha)}{2n_0}} \\[1mm]
\mbox{subject to} &\quad z \in \ConfEll\\[1mm]
\end{split}
\end{equation}
This problem is not convex, and because of the cap with $\varrho$ ($\land$ denotes ``min''), we cannot rely on an equivalent convex dual problem, as in the case of the uncapped variant \cite{csaji2023improving,horvath2023nonparametric}. However, we do not need to solve \eqref{noisy-norm-max-rewritten} exactly, an upper bound is enough, which can be easily given, see Remark \ref{remark-tau-computation}.
Let us denote the {\em optimal value} of \eqref{noisy-norm-max-rewritten} by $\tau_0$, then
\smallskip
\begin{lemma}
\label{lemma-norm-estimation-known-density-noisy}
{\em Assume that A\ref{assumption-iid-multivariate}, A\ref{assumption-positive-input-density}, 
A\ref{assumption-Paley-Wiener-space-multivariate} and A\ref{assumption-ellipsoid-guarantee} hold. Then, for any $\alpha,\beta \in (0,1)$ risk probabilities, we have
$$\mathbb{P}\big((f_*(x_1), \dots, f_*(x_{n_0}))\tr\hspace{-0.5mm} \in \ConfEll\hspace{0.1mm} \land\hspace{0.1mm} \norm{f_*}_{\CH}^2 \leq \tau_0 \hspace{0.1mm}\big) \geq 1-\alpha-\beta.
$$}
\end{lemma}
\smallskip
\begin{proof}
This directly follows from the combination (i.e., union bound) of A\ref{assumption-ellipsoid-guarantee} and Lemma \ref{lemma-norm-estimation-known-density-noisefree}, using that $\tau_0$ solves \eqref{noisy-norm-max-rewritten} and that the cap is {\em inactive} at the true outputs by A\ref{assumption-Paley-Wiener-space-multivariate}; so the objective of \eqref{noisy-norm-max-rewritten} at $z_*$ is exactly the bound $\kappa_{n_0}$ of Lemma \ref{lemma-norm-estimation-known-density-noisefree}.
\end{proof}
\vspace{-1mm}

Recall that $\beta$ comes from the uncertainty of the ellipsoids $\ConfEll$, and $\alpha$ comes from the uncertainty of the norm estimation.

\subsection{Confidence Bands and Interval Endpoints}
Now, we can define our {\em abstract} confidence sets as
\begin{equation*}
\begin{aligned}
\mathcal{D}_{n}\hspace{-0.3mm} &\defeq\hspace{-0.3mm} \{ f \in \CH \mid \exists\, z \in \ConfEll\hspace{-0.3mm}:\hspace{-0.3mm} f \interpolates \{(x_k,z_k)\}_{k=1}^{n_0} \land \norm{f}_\mathcal{H}^2 \leq \tau_{0} \}.
\end{aligned}
\end{equation*}
This construction guarantees that $\mathbb{P}(f_* \in \mathcal{D}_{n}) \geq 1-\alpha-\beta$, since with probability at least $1-\alpha-\beta$ we have that the true outputs $\{f_*(x_i)\}$,  for $i \in [n_0]$, are in the ellipsoid $\ConfEll$ and $\norm{f_*}_{\CH}^2 \leq \tau_0$, by Lemma \ref{lemma-norm-estimation-known-density-noisy}, moreover, as $f_*$ interpolates its own outputs, $f_* \in \mathcal{D}_{n}$. However, as this direct construction is computationally challenging, an algorithm is needed to efficiently generate the confidence set for an arbitrary input.

Similarly to the noise-free case, confidence band $\tilde{\mathcal{B}}_n$ can be built by calculating an envelope for the possible outputs:
\begin{equation}
\label{eq:confband-noisy}
\tilde{\mathcal{B}}_n \defeq \big\{ f \in \CH \mid \forall x\in \mathbb{R}^d: \tilde{L}_n(x) \leq f(x) \leq \tilde{U}_n(x) \big\},
\end{equation}
where the {\em interval endpoints} for input $x \in \mathbb{R}^d$ are defined by
\begin{equation*}
\tilde{L}_n(x) \defeq \inf_{f \in \mathcal{D}_{n}}\! f(x),\qquad \text{and}\qquad \tilde{U}_n(x) \defeq \sup_{f \in \mathcal{D}_{n}}\! f(x).
\end{equation*}

Now, we present the construction for the interval endpoints, given query input $x_0 \in \mathbb{R}^d$, with $x_0 \neq x_k$, for $k \in [n_0]$.
\begin{enumerate}
    \item[(i)] First, the extended Gram matrix should be built 
    $$\tilde{K}_0(i+1,j+1) \defeq k(x_i,x_j),$$
    for $i, j = 0,1, \dots, n_0$. Note that the dimension of $\tilde{K}_0$ is $(n_0+1) \times (n_0+1)$, as it also includes query input $x_0$.
\item[(ii)] Similarly to the noise-free case, we solve two convex optimization problems with quadratic constraints to get $\tilde{B}_n(x_0) = (\tilde{L}_n(x_0),\tilde{U}_n(x_0))$, but since now we do not observe the outputs that should be interpolated, we optimize over all possible vectors, $(z_1, \dots, z_{n_0})\tr \in \ConfEll$:    
    \begin{align}
    \mbox{min\,/\,max} &\quad z_{0}\notag \\[0.5mm]
    \mbox{subject to} &\quad (z_0, \dots, z_{n_0})\hspace{0.3mm} {\tilde{K}}_0^{-1} (z_0, \dots, z_{n_0})\tr \leq\, \tau_0\notag\\[1mm]
    \label{noisy-opt-min-max-modified}
    &\quad (z_1, ..., z_{n_0}) \in \ConfEll,
    \end{align}
    where ``min\,/\,max'' again means that the problem has to be solved as a minimization and as a maximization.%
    \item[(iii)] The optimal values, denoted by $z_{\mathrm{min}}$ and $z_{\mathrm{max}}$, are the {\em endpoints} of the {\newpart confidence} interval: $\tilde{L}_n(x_0) \defeq z_{\mathrm{min}}$, and $\tilde{U}_n (x_0) \defeq z_{\mathrm{max}}$. If \eqref{noisy-opt-min-max-modified} is infeasible, e.g., we get an empty KGP ellipsoid, we return $\tilde{B}_n(x_0) = \emptyset$.
\end{enumerate}
\smallskip
The coverage guarantee of the construction is
\smallskip
\begin{theorem}
\label{theorem-reliability-noisy}
{\em Assume that A\ref{assumption-iid-multivariate}, A\ref{assumption-positive-input-density}, 
A\ref{assumption-Paley-Wiener-space-multivariate}, and A\ref{assumption-ellipsoid-guarantee} are satisfied. Let $\alpha, \beta \in (0,1)$ be given risk probabilities.
Then, for any $n \in \mathbb{N}$, the confidence band described above guarantees}
$$\nu(\tilde{B}_n)\, =\, \mathbb{P}(f_* \in \tilde{\CB}_n)\, \geq\, 1-\alpha - \beta.$$
\end{theorem}
\begin{proof}
Let us introduce the following event:
$$A \defeq \{ (f_*(x_1), \dots, f_*(x_{n_0}))\tr\hspace{-0.5mm} \in \ConfEll \land \norm{f_*}_\CH^2 \leq \tau_0 \}$$
According to Lemma \ref{lemma-norm-estimation-known-density-noisy}, the event $A$ has probability at least $1-\alpha-\beta$.  Conditioning on $A$, for all $x_0 \in \mathbb{R}^d$ such that $\tilde{K}_0$ is invertible (which holds $\mathbb{P}_{\!\scriptscriptstyle{X}}$-a.s., as the Paley--Wiener kernel is strictly positive definite) we know that there is a $z \in \ConfEll$, namely $z = (f_*(x_1), \dots, f_*(x_{n_0}))\tr$, and $z_0$, namely $z_0 = f_*(x_0)$, such that the {\em minimum-norm} interpolant of $\{(x_0, z_0)\} \cup \{(x_k, z_k)\}_{k=1}^{n_0}$ has a norm square $\leq \tau_0$, since $f_*$ itself is an interpolant of this dataset.
Thus, we have that $z_{\mathrm{min}} \leq z_0 \leq z_{\mathrm{max}}$, for $z_0 = f_*(x_0)$. This property is always guaranteed under event $A$, hence, we {\em simultaneously} have for all input $x_0 \in \mathbb{R}^d$ that $z_{\mathrm{min}}(x_0) \leq f_*(x_0) \leq z_{\mathrm{max}}(x_0)$.
\end{proof}
\vspace*{-6mm}

\subsection{Strong Uniform Consistency}
In this section, we prove that MiNCE is strongly uniformly consistent under measurement noises. Our strategy is again to analyze  a confidence band construction which provides strictly larger confidence bands, since it will be easier to study and its consistency implies the consistency of MiNCE, as well.

Thus, instead of ellipsoid $\ConfEll$, in the proof we work with its induced individual confidence intervals, that is, for $k \in [\hspace{0.3mm}n_0\hspace{0.3mm}]$
\vspace{-0.5mm}
\begin{equation}
\label{noisy-norm-max-intervals}
\begin{aligned}
\nu_{n,k} &\,\defeq \, \min\hspace{-0.3mm} \big\{\hspace{0.3mm} z_k : z \in \ConfEll\hspace{0.3mm}\big\}
= e_k\tr\hat{z}_n - (e_k\tr P_n^{-1} e_k)^{\frac{1}{2}},\\[2mm]
\mu_{n,k} &\,\defeq \, \max\hspace{-0.3mm} \big\{\hspace{0.3mm} z_k : z \in \ConfEll\hspace{0.3mm}\big\} = e_k\tr\hat{z}_n + (e_k\tr P_n^{-1} e_k)^{\frac{1}{2}},%
\end{aligned}
\vspace{1mm}
\end{equation}
where $e_k$ is the $k$-th canonical orthonormal basis vector. 

By construction we have $\ConfEll \subseteq \ConfHyp \doteq \times_{k=1}^{n_0} \: [\nu_{n,k}, \mu_{n,k}]$. Combining this with A\ref{assumption-ellipsoid-guarantee} yields the coverage guarantee
\begin{equation}
\label{sps-intervals-inputs}
\mathbb{P}\big(\hspace{0.3mm} \forall \hspace{0.3mm}k \in [\hspace{0.3mm}n_0\hspace{0.3mm}]: f_*(x_k) \in [\hspace{0.3mm}\nu_{n,k}, \mu_{n,k}\hspace{0.3mm}]\hspace{0.3mm}\big)\, \geq\, 1 - \beta,
\end{equation}
which implies that the confidence intervals $\{[\hspace{0.3mm}\nu_{n,k}, \mu_{n,k}\hspace{0.3mm}]\}$ have {\em simultaneous} coverage for the true outputs, $\{f_*(x_k)\}$.

We also introduce a new norm upper bound $\tau_n$, which is obtained by optimizing over $\ConfHyp$, instead of $\ConfEll$, that is
\begin{equation}
\label{noisy-norm-bound-with-intervals}
\begin{split}
\mbox{maximize} &\quad \frac{1}{n_0}\, \sum_{k=1}^{n_0} \left(\frac{{z_k}^2}{h_*(x_k)} \land \varrho\right) + \varrho \, \sqrt{\frac{\ln (1/\alpha)}{2n_0}} \\[1mm]
\mbox{subject to} &\quad z \in \ConfHyp \doteq \mathop{\times}\limits\limits_{k=1}^{n_0}[\nu_{n,k}, \mu_{n,k}]\\[0.5mm]
\end{split}
\end{equation}
Its optimal value is $\tau_n$, which can be explicitly written as
\begin{equation}
\label{norm-estimation-noisy}
\tau_n \, =\, \frac{1}{n_0} \sum_{k=1}^{n_0} \left(\frac{\max\{\nu_{n,k}^2, \mu_{n,k}^2\}}{h_*(x_k)} \land \varrho\right) + \varrho\, \sqrt{\frac{\ln (1/\alpha)}{2n_0}}.
\end{equation}
We use $n$ as its index (not $n_0$), since $\{\nu_{n,k}\}$ and $\{\mu_{n,k}\}$ are based on $\ConfEll$ that was built using the whole sample $\CZ_n$. Then\!\!\!
\smallskip
\begin{lemma}
\label{noisy-improved-version-subset}
{\em Let $\tau_0$ and $\tau_n$ be the optimal values of optimization problems \eqref{noisy-norm-max-rewritten} and \eqref{noisy-norm-bound-with-intervals}, respectively. We have $\tau_0 \leq \tau_n$.}
\end{lemma}
\smallskip
\begin{proof}
$\tau_0$ and $\tau_n$ are optimal values of constrained maximization problems with the same objective function. For $\tau_0$ we optimize over ellipsoid $\ConfEll$, while for $\tau_n$ we optimize over hypercube
$\ConfHyp.$
Since $\ConfEll \subseteq \ConfHyp$, we get $\tau_0 \leq \tau_n$.
\end{proof}
\begin{remark}
\label{remark-tau-computation}
{\em As $\tau_0$ is only used as an {\em upper bound} for $\norm{f_*}_\CH^2$, it can be replaced everywhere by any $\tau$ with $\tau_0 \leq \tau \leq \tau_n$; larger values only widen the bands. As the cap can only decrease the objective, the optimal value $\tau_0^{\circ}$ of the {\em uncapped} problem, available from the convex dual, is such an upper bound, hence $\tau \defeq \min\{\tau_0^{\circ}, \tau_n\}$ can always be used.}
\end{remark}
\smallskip
We introduce another assumption stating that the underlying confidence ellipsoids shrink as the sample size increases.
\smallskip
\begin{assumption}
\label{assumption-ellipsoid-consistency}
{\em Let $\ConfEll$ be the ellipsoid \eqref{eq:conf-ell-set}, with center $\hat{z}_n$, and let $z_* \defeq (f_*(x_1), \dots, f_*(x_{n_0}))\tr$. We assume $n_0(n) \to \infty$, as $n\to \infty$, moreover, that $\ConfEll$ shrinks and its center is strongly consistent, that is, $\sup_{z_1, z_2 \in \mathcal{E}_{n_0}^n} \|z_1-z_2\hspace{0.3mm}\|_{K_{n_0}^{-1}}\convergealmostsurely 0$, as $n \to \infty$, and $\|\hat{z}_n-z_*\|_{K_{n_0}^{-1}}\convergealmostsurely 0$, where $\|v\|_{K_{n_0}^{-1}} \defeq (v\tr K_{n_0}^{-1} v)^{1/2}$ and $K_{n_0}$ is the kernel matrix w.r.t.\ inputs $x_1, \dots, x_{n_0}.$}
\end{assumption}
\smallskip
\begin{remark}
\label{remark-ellipsoid-consistency}
{\em As $\ConfEll$ is symmetric about $\hat{z}_n$, its radius is half of its diameter, hence, by the triangle inequality, A\ref{assumption-ellipsoid-consistency} implies
\begin{equation}
\label{eq:A5-mahalanobis}
    \sup_{z \in \mathcal{E}_{n_0}^n}\, (z-z_*)\tr K_{n_0}^{-1}(z-z_*)\convergealmostsurely\, 0,
\end{equation}
as $n \to \infty$, which is the form used in our proofs; note that it does {\em not} require $z_* \in \ConfEll$. Its second part is consistency in the RKHS norm, as $\|\hat{z}_n-z_*\|_{K_{n_0}^{-1}} = \|\hat{f}_{n,\hat{z}_n}-\hat{f}_{n,z_*}\|_\CH$; the Euclidean norm would not suffice, since $\lambda_{\min}(K_{n_0}) \to 0$.}
\end{remark}

With this, the norm bound $\tau_n$ is also strongly consistent:
\smallskip
\begin{lemma}
\label{norm-consistency-2}
{\em Assuming A\ref{assumption-iid-multivariate}, A\ref{assumption-positive-input-density}, 
A\ref{assumption-Paley-Wiener-space-multivariate},  A\ref{assumption-ellipsoid-guarantee}, and A\ref{assumption-ellipsoid-consistency}, we have
$$\tau_{n} \convergealmostsurely \norm{f_*}_\CH^2 = \kappa_*, \qquad\text{as}\;\, n \to \infty.$$
\vspace{-4mm}}
\end{lemma}

The proof of this lemma can be found in Appendix \ref{appendix-norm-consistency-2}.

Let $z_* \defeq (f_*(x_1), \dots, f_*(x_{n_0}))\tr$ be the vector of noise-free outputs, let ${\tau}_{n}^* \defeq \max \{ \tau_{n}, \kappa_* \}$ and $\ConfEllStar = \ConfEll\cup\{z_*\}$. As a theoretical tool for our consistency analysis, we introduce
\begin{equation*}
\begin{aligned}
{\mathcal{D}}^{*}_{n} &\defeq \{ f \in \CH \mid \exists\, z \in \ConfEllStar\hspace{-0.3mm}:\hspace{-0.3mm} f \interpolates \{(x_k,z_k)\}_{k=1}^{n_0} \land \norm{f}_\mathcal{H}^2 \leq {\tau}_{n}^* \}.
\end{aligned}
\end{equation*}

Note that since $\ConfEll \subseteq \ConfEllStar$ and $\tau_0\leq  {\tau}_{n}^*$, we have $\mathcal{D}_n \subseteq {\mathcal{D}}^{*}_{n}$. 
Thus, it is enough to show the (unif.) consistency of $\{{\mathcal{D}}^{*}_{n}\}$ as it implies the consistency of $\{\mathcal{D}_{n}\}$, from which the consistency of the confidence bands $\{\tilde{B}_n\}$ follows, using Theorem \ref{theorem-consistency-abstract-endpoint-based}.

The minimum-norm interpolant of any $z \in \ConfEllStar$ is
$$\hat{f}_{n,z} \defeq \argmin \big\{\,\|\hspace{0.3mm}f\hspace{0.4mm}\|_{\mathcal{H}} \mid f \in \mathcal{H}\hspace{1.5mm} \&\hspace{1.5mm} \forall\hspace{0.3mm} k \in [n_0]: f(x_k) =\, z_k   \,  \big\}.$$
Using this notation, we write $\hat{f}_{n,z_*}$ to denote the minimum-norm interpolant of the true outputs $z_*$. Let $ \hat{\tau}_{n} \doteq \|\hat{f}_{n,z_*}\|^2_{\CH}$. 

\smallskip
\begin{remark}
{\em As $\hat{f}_{n,z_*}$ and $f_{*}$ both interpolate $z_*$, and because $\hat{f}_{n,z_*}$ has the smallest 
norm among such functions, we have
$$\|\hat{f}_{n,z_*}\|^2_\CH \,\leq \,\norm{f_{*}}^2_\CH \,=\, \kappa_* \,\leq\, \tau_n^*,$$
therefore, $f_*$ and $\hat{f}_{n,z_*}$ are both elements of\, ${\mathcal{D}}^{*}_{n}$.}
\end{remark}
\smallskip

A key observation for the strong uniform consistency theorem is summarized by the lemma below.

\smallskip
\begin{lemma}
\label{ellipsoid-norm-diff}
{\em Assume A\ref{assumption-iid-multivariate}, A\ref{assumption-positive-input-density}, 
A\ref{assumption-Paley-Wiener-space-multivariate},  A\ref{assumption-ellipsoid-guarantee}, and A\ref{assumption-ellipsoid-consistency}. Let $f,f' \in \mathcal{D}^{*}_{n}$ such that they interpolate $z, z' \in \ConfEllStar$, respectively. Then
$$\|f-f'\|_\CH^2 \leq (z - z')\tr K_{n_0}^{-1} (z - z') +\vspace{-1mm}$$
$$\left( \sqrt{\tau^*_n - z\tr K_{n_0}^{-1}z} + \sqrt{\tau^*_n - (z')\tr K_{n_0}^{-1} z'} \,\right)^{\!2}\!\!,\vspace{-1mm}$$
where $K_{n_0}$ is the Gram matrix associated with $x_1, \dots, x_{n_0}.$}
\end{lemma}
\smallskip
The proof of Lemma \ref{ellipsoid-norm-diff} can be found in Appendix \ref{appendix-ellipsoid-norm-diff}.

The last result that we need is that the norms of the minimal-norm interpolants of all output vectors in the confidence ellipsoid converge to the norm of the regression function:

\smallskip
\begin{lemma}
\label{ellipsoid-norm-convergence}
{\em Assume A\ref{assumption-iid-multivariate}, A\ref{assumption-positive-input-density}, 
A\ref{assumption-Paley-Wiener-space-multivariate},  A\ref{assumption-ellipsoid-guarantee}, and A\ref{assumption-ellipsoid-consistency}. Then
$$
\sup_{z \in \ConfEllStar}\! z\tr K_{n_0}^{-1} z \convergealmostsurely  \kappa_*,\quad\text{and}\quad
\inf_{z \in \ConfEllStar}\! z\tr K_{n_0}^{-1} z \convergealmostsurely \kappa_*,
$$
as $n\to \infty$, where $\kappa_* =\|f_*\|_{\CH}^2$.}
\end{lemma}
\smallskip
From our technical lemmas the {\em strong uniform consistency} of the {\em abstract} MiNCE confidence regions $\{\mathcal{D}_n\}$ follows:

\smallskip
\begin{theorem}
\label{consistency-known-density}
{\em Assume A\ref{assumption-iid-multivariate}, A\ref{assumption-positive-input-density}, 
A\ref{assumption-Paley-Wiener-space-multivariate},  A\ref{assumption-ellipsoid-guarantee}, and A\ref{assumption-ellipsoid-consistency}. Then
$$
\sup_{f\in\mathcal{D}_n}\!\|f - f_*\|_{\mathcal{H}}\,\convergealmostsurely\;0, \quad\text{as}\quad n \to \infty.
$$
}
\end{theorem}
\begin{proof}
Because $\mathcal{D}_n \subseteq \mathcal{D}_n^*$ by construction, it is enough to prove the strong uniform consistency for sets $\{\mathcal{D}_n^*\}$.

By applying Lemma \ref{ellipsoid-norm-diff}, we have
\vspace{-1mm}
\begin{align*}
    &\sup_{f \in \mathcal{D}_n^*} \|f - f_*\|_{\mathcal{H}}^2 \le \overbrace{\sup_{z \in \mathcal{E}_{n_0}^*} (z - z_*)\tr K_{n_0}^{-1} (z - z_*)}^{\xi_{1,n}} 
    \\[-1mm]
    &\quad + \bigg( \underbrace{\sup_{z \in \mathcal{E}_{n_0}^*}\!\! \sqrt{\tau_n^* - z\tr K_{n_0}^{-1} z}}_{\xi_{2,n}} + \underbrace{\sqrt{\tau_n^* - z_*\tr K_{n_0}^{-1} z_*}}_{\xi_{3,n}} \bigg)^{\!2}\!\!.
\end{align*}
\vspace{-2mm}

We study the asymptotics of $\xi_{1,n},\xi_{2,n}$ and $\xi_{3,n}$ separately.

Since $(z - z_*)\tr K_{n_0}^{-1} (z - z_*) \geq 0,$ for all $z \in \mathbb{R}^{n_0}$, the supremum over $\mathcal{E}_{n_0}^* \doteq \mathcal{E}_{n_0}^n \cup \{z_*\}$ simplifies to the supremum over $\ConfEll$, therefore 
$\xi_{1,n}\convergealmostsurely 0$, as $n\to \infty$, 
by \eqref{eq:A5-mahalanobis}.

As for the term $\xi_{2,n}$, observe that
\begin{equation*}
        \sup_{z \in \ConfEllStar}\! \sqrt{\tau_n^* - z\tr K_{n_0}^{-1} z} = \sqrt{\tau_n^* - \inf_{z \in \ConfEllStar}\! z\tr K_{n_0}^{-1} z}.
        \vspace{-1mm}
\end{equation*}
We know by Lemma \ref{ellipsoid-norm-convergence} that $\inf_{z \in \ConfEllStar}\! z\tr K_{n_0}^{-1} z \convergealmostsurely \kappa_*$ and, by Lemma \ref{norm-consistency-2}, that $\tau_n^* \xrightarrow{\text{a.s.}} \kappa_*$. Thus, $\xi_{2,n} \convergealmostsurely 0$, as $n \to \infty$.

Finally, by Lemma \ref{ellipsoid-norm-convergence}, $z_*\tr K_{n_0}^{-1} z_* \xrightarrow{\text{a.s.}} \kappa_*$ and, by Lemma \ref{norm-consistency-2}, $\tau_n^* \xrightarrow{\text{a.s.}} \kappa_*$, implying, $\xi_{3,n} \convergealmostsurely \sqrt{\kappa_* - \kappa_*} = 0$, as $n \to \infty$.

Hence, $\sup_{f \in \mathcal{D}_n^*} \|f - f_*\|_{\mathcal{H}} \convergealmostsurely 0$, completing the proof. 
\end{proof}

Combining Theorems \ref{theorem-consistency-abstract-endpoint-based} and \ref{consistency-known-density}, we can conclude that the MiNCE confidence {\em bands} are also {\em strongly uniformly consistent} in case there are measurement noises, that is,
\smallskip
\begin{corollary}
{\em Under the assumptions of Theorem \ref{consistency-known-density}, we have
\begin{equation*}
\max\hspace{-0.5mm}\big\{\hspace{0.5mm} \|\tilde{L}_n - f_*\|_\infty, \|\hspace{0.3mm} \tilde{U}_n - f_*\|_\infty \big\} \convergealmostsurely\, 0, \quad\text{as}\quad n\to\infty.
\end{equation*}}
\end{corollary}

\section{Robust Spectral Estimation}

Now, we demonstrate how MiNCE can be applied to spectral estimation, a fundamental problem in signal processing.
The task is to estimate the (magnitude and\;/\,or phase) {\em spectrum} of a signal \cite{tary2014spectral}. In case of the magnitude spectrum, one of the most prominent approaches is the {\em periodogram}, which provides an {\em asymptotically unbiased} estimate of the spectrum using the discrete Fourier transform; however, it is {\em not consistent}, since its variance does not converge to zero as the sample size increases \cite{stoica2005spectral}. Even for noise-free outputs, its core assumption is deterministic, equidistant sampling which may be unrealistic for some applications (like wireless sensor networks, compressive sensing, mobile sensing, etc.). As we consider randomly sampled (non-equidistant) inputs, approaches like the periodogram cannot be applied in our case.

Here, we focus on constructing a consistent point estimate and corresponding {\em nonasymptotic confidence regions} for the spectrum, based on a random sample of input-output pairs. Note that we cannot directly use our previous construction in the frequency domain, since (i) we do not have direct observations of the spectrum, and, moreover, (ii) $\mathcal{L}^2([-\eta,\eta]^d)$ is not even an RKHS, point evaluations are not continuous.

Nevertheless, an abstract, nonasymptotic MiNCE confidence region, which contains functions in the spatial domain, can be transformed to a confidence region in the frequency domain, simply by including the Fourier transform of each function in the original (spatial domain) set. However, the obtained abstract sets are hard to work with in practice. We will show how to use MiNCE to get a strongly consistent estimate of the {\em smoothed} spectrum of a signal, based on a random dataset. Smoothing can play an important role, as it reduces the variance of the estimator and improves the consistency guarantees \cite{colbrook2021computing}. For simplicity, we will focus on the case of {\em noise-free} outputs (at random inputs).

As the minimum-norm interpolant, $\hat{f}_n$, is a key object for MiNCE, we start by computing its Fourier transform,
\begin{equation*}
(\mathcal{F}(\hat{f}_n))(\omega) = \frac{1}{(2\pi)^{d/2}} \int_{\RR^d} \hat{f}_n(x)\, e^{-i\langle x,\omega\rangle}\: \mbox{d}x,
\end{equation*}
for $\omega \in [-\eta,\eta]^d$, otherwise $0$. By using the special structure of the minimum-norm interpolant, see \eqref{min-norm-interpolant-formula}, we have
\begin{equation*}
\begin{aligned}
(\mathcal{F}(\hat{f}_n))(\omega) &= \sum_{k=1}^n \hat{\alpha}_k \mathcal{F}(k(\cdot,x_k))(\omega)\\
&= \sum_{k=1}^n \hat{\alpha}_k \frac{e^{-i\langle\omega, x_k\rangle}}{(2\pi)^{d/2}} \mathbb{I}_{[-\eta,\eta]^d}(\omega),
\end{aligned}
\end{equation*}
where $\mathbb{I}_A$ is the indicator function of set $A$. In the noise-free case, under our assumptions, $\mathcal{F}(\hat{f}_n)$ is a strongly consistent estimator of $\mathcal{F}(f_*$), since $\| \hat{f}_n - f_*\|_2 \convergealmostsurely 0, 
\text{\;as\;}n\to \infty, \implies \| \mathcal{F}(\hat{f}_n) - \mathcal{F}(f_*)\|_2 \convergealmostsurely 0$, due to Plancherel's theorem.

For finite, randomly sampled datasets, it is more efficient and robust to estimate the {\em smoothed} spectrum. 
Let $f \in \CH$,
the {\em smoothed spectrum} of $f$ at frequency vector $\omega \in [-\eta,\eta]^d$ is
\begin{equation*}
\begin{aligned}
(\CS_{\phi}(f))(\omega) &\defeq (\mathcal{F}(f) \star \phi)(\omega)
\\ &=\int_{[-\eta,\eta]^d} (\mathcal{F}(f))(\gamma) \phi(\omega-\gamma) \:\mbox{d}\gamma,
\end{aligned}
\end{equation*}
where $\star$ denotes the convolution operator and function $\phi \in \mathcal{L}^2([-\eta,\eta]^d)$ is a user-chosen nonnegative {\em smoothing} or {\em window} function which is $\mathcal{L}^1$ normalized,  
$\int_{[-\eta,\eta]^d} \phi(\omega) \: \mbox{d}\omega = 1$.

A typical smoother is the rectangular function. Given (hyperparameter) half-width $h=(h_1, \dots, h_d)$, it is defined as\vspace{-1.5mm}
$$
\phi_h(\omega) = \frac{1}{\prod_{i=1}^d(2h_i)}\prod_{i=1}^d \mathbb{I}_{[-h_i,h_i]}(\omega_i).
$$
A more sophisticated choice is the triangular smoother, also called Fejér or Bartlett window, which is defined as\vspace{-1mm}
$$
\phi_h(\omega) = \prod_{i=1}^d\frac{1}{h_i}\bigg( 1- \frac{|\omega_i|}{h_i}\bigg)\mathbb{I}_{[-h_i,h_i]}(\omega_i),
$$
for which $\left\| \phi_h \right\|_1 = 1$, moreover, one can easily show that
\vspace{-1mm}
\[
\left\| \phi_h \right\|_2
=
\left( \prod_{i=1}^d \frac{2}{3 h_i} \right)^{\!\!1/2}\!\!\!\!\!
=
\left( \frac{2}{3} \right)^{\!d/2}\!
\left( \prod_{i=1}^d h_i \right)^{\!\!-1/2}\!\!\!\!\!\!\!\!\!\!\!.
\]
Now, let us take a look at how well the smoothed spectrum of the minimum-norm interpolant approximates the smoothed spectrum of the original function. Observe that, for all $\omega$,
\begin{align*}
&|\,(\CS_{\phi}(\hat{f}_n))(\omega)- (\CS_{\phi}(f_*))(\omega)\,|=\\
&\bigg|\int_{[-\eta,\eta]^d} \Big((\mathcal{F}(\hat{f}_n))(\gamma) - (\mathcal{F}(f_*))(\gamma)\Big) \phi(\omega-\gamma) \:\mbox{d}\gamma\,\bigg|.
\end{align*}
Then, by using the Cauchy-Schwarz inequality, we have
\begin{align*}
&\sup\nolimits_{\omega}|\,(\CS_{\phi}(\hat{f}_n))(\omega)- (\CS_{\phi}(f_*))(\omega)\,|\\
& \leq\, \| \CF(\hat{f}_n) - \CF(f_*)\|_2 \; \|\phi\|_2 \\
& \leq\, \| \hat{f}_n - f_*\|_2 \; \|\phi\|_2 \convergealmostsurely\hspace{0.3mm} 0,
\end{align*}
as $n\to \infty$, by Plancherel’s theorem, therefore, $\CS_{\phi}(\hat{f}_n)$ is a {\em strongly uniformly consistent} estimate of $\CS_{\phi}(f_*)$, assuming the min-norm interpolant converges, e.g., under A\ref{assumption-iid-multivariate}, A\ref{assumption-positive-input-density} and A\ref{assumption-Paley-Wiener-space-multivariate}.

Then, we can build an {\em abstract} nonasymptotic confidence region for $\CS_{\phi}(f_*)$ based on the MiNCE framework, as follows:
\begin{equation*}
\mathcal{Q}_n \doteq \big\{\,\CS_{\phi}(f): f \in \CH\, \land\, \| \hat{f}_n - f\|_2 \leq r_n \hspace{0.3mm}\big\},
\end{equation*}
where $r_n \doteq \sqrt{|\kappa_n - \hat{\kappa}_n|_+}$ is the radius from \eqref{eq:outer-ball}; that is, $\mathcal{Q}_n$ is the image of $\mathcal{O}_n$ under $\CS_{\phi}$. Since $\|\phi\|_1 = 1$, Young's inequality gives $\|\CS_{\phi}(f)-\CS_{\phi}(\hat{f}_n)\|_2 \leq r_n$, for $f \in \mathcal{O}_n$, while the {\em pointwise} deviations 
are at most $\tilde{r}_n \doteq r_n \cdot \|\phi\|_2$, by the Cauchy-Schwarz inequality.

An immediate consequence of Corollary \ref{thn:outer-ball} and the {\em continuity} of the (linear) operator $\CS_{\phi}: \CH \to \CL^2([-\eta,\eta]^d)$ is
\smallskip
\begin{corollary}
\label{thn:spectrum-ball}
{\em Under the assumptions of Theorem \ref{consistency-noisefree-known-density}, we have
\begin{equation*}
\mathbb{P}(\CS_{\phi}(f_*) \in \mathcal{Q}_n) \geq 1-\alpha,\qquad
\sup_{g\in\mathcal{Q}_n}\!\!\|g - \CS_{\phi}(f_*)\|_2\econvergealmostsurely\,0.\vspace{-1.5mm}
\end{equation*}}
\end{corollary}
\smallskip
Therefore, our induced MiNCE confidence bands $\{\mathcal{Q}_n\}$ for the smoothed spectrum $\CS_{\phi}(f_*)$ have guaranteed {\em nonasymptotic coverage} and they are also {\em strongly uniformly consistent}.

Such confidence region for the smoothed spectrum can be easily transformed to confidence regions for the (smoothed) magnitude and phase spectra. Recall that the {\em magnitude} and {\em phase} spectra of $f \in \CL^2$ at frequency $\omega$ are defined as
\begin{align*}
(\mathcal{M}_{\phi}(f))(\omega)\!\doteq\!|(\mathcal{S}_{\phi}(f))(\omega)|,\;
(\mathcal{P}_{\phi}(f))(\omega)\!\doteq\!\text{Arg}((\mathcal{S}_{\phi}(f))(\omega)),
\end{align*}
where ``$\text{Arg}(z)$'' denotes the principal argument of $z \in \mathbb{C}$.

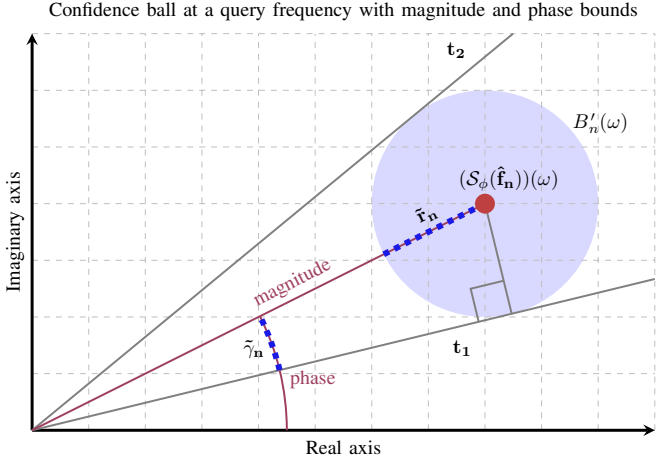
\begin{figure}[!t]
    \centering
    \resizebox{\columnwidth}{!}{%
        \input{figtikz2.tex}%
    }
    \vspace{-5mm}
\caption{Confidence ball $B\hspace{0.3mm}'_{\hspace{-0.3mm}n}(\omega)$ at query frequency $\omega$ around $(\mathcal{S}_{\phi}(\hat{f}_n))(\omega)$. The confidence intervals for the true magnitude and phase can be determined by the extrema of magnitudes and phases on this ball, see $\tilde{r}_{n}$ and $\tilde{\gamma}_n$.}
    \label{fig:phase-confidence}
    \vspace*{-2mm}
\end{figure}

For easier practical applicability, a pointwise version is,
\begin{equation*}
B'_n(\omega) \doteq \big\{\, z \in \BC: |\hspace{0.3mm}z-(\mathcal{S}_{\phi}(\hat{f}_n))(\omega)\hspace{0.3mm}| \leq \tilde{r}_n \hspace{0.3mm}\big\},
\end{equation*}
from which one can define confidence bands for the smoothed {\em magnitude} ($\mathcal{M}$) and {\em phase} ($\mathcal{P}$) spectra, that is, for a given $\omega$,
\begin{align*}
B_n^{\mathcal{M}}(\omega) &\doteq \big[|(\mathcal{S}_{\phi}(\hat{f}_n))(\omega)|-\tilde{r}_n, |(\mathcal{S}_{\phi}(\hat{f}_n))(\omega)|+\tilde{r}_n\big],\\[1mm]
B_n^{\mathcal{P}}(\omega) &\doteq \big[\text{Arg}((\mathcal{S}_{\phi}(\hat{f}_n)(\omega))-\tilde{\gamma}_n, \text{Arg}((\mathcal{S}_{\phi}(\hat{f}_n))(\omega))+\tilde{\gamma}_n\big],
\end{align*}
where $\tilde{\gamma}_n \doteq \arcsin\!\big(\tilde{r}_n / |(\mathcal{S}_{\phi}(\hat{f}_n))(\omega)|\big),$ if we have $\tilde{r}_n < |(\mathcal{S}_{\phi}(\hat{f}_n))(\omega)|$, see Figure \ref{fig:phase-confidence} for a geometric intuition, otherwise $\tilde{\gamma}_n = \infty$. This latter case can be explained by the fact, that if $0 \in B_n'(\omega)$, then any phase value is achievable. We can, of course, also intersect $B_n^{\mathcal{M}}(\omega)$ with $[\hspace{0.3mm}0,\infty)$ and $B_n^{\mathcal{P}}(\omega)$ with $(-\pi, \pi]$, the range of the argument function.

Finally, since $B_n^{\mathcal{M}}$ and $B_n^{\mathcal{P}}$ are outer approximations, they have {\em finite sample coverage} guarantees. Furthermore, since $\tilde{r}_n \convergealmostsurely 0$, as $n\to \infty$, they are {\em strongly uniformly consistent}.

\section{Numerical Experiments}
The presented MiNCE framework was also validated numerically. We investigated the univariate case $(d=1)$. The (bandwidth) parameter of the Paley--Wiener RKHS was set to $\eta = 30$ in the noise-free case and to $\eta = 20$ in case of measurement noises. 
The true data-generating function $f_*$ was constructed as follows: first, $20$ random inputs $\{\bar{x}_k\}_{k=1}^{20}$ were sampled using the uniform distribution on $[\hspace{0.3mm}-1,1]$. Then, $f_*(x) = \sum_{k=1}^{20} w_k k(x, \bar{x}_k, \eta_k')$ was created, where each $w_k$ followed a standard normal distribution, $k(\cdot, \cdot, \eta')$ denotes the Paley--Wiener kernel with parameter $\eta'$, and the values of $\{\eta_k'\}$ were chosen uniformly from $[\hspace{0.3mm}0.1, \eta\hspace{0.3mm}]$, hence, $\{k(\cdot,\bar{x}_k,\eta_k')\}$ were all members of the Paley--Wiener RKHS with parameter $\eta$. Finally, $f_*$ was normalized (rescaled) to have $\|f_*\|_\infty = 1$. This construction was chosen to make the problem more complex, as it (a.s.) needs infinitely many coefficients in terms of $k(\cdot,\cdot,\eta)$ to encode $f_*$. Since $|f_*(x)| = \mathcal{O}(1/|x|)$, the ratio $f_*^2/h_*$ is bounded on $\mathbb{R}$ for Cauchy input distributions, therefore A\ref{assumption-Paley-Wiener-space-multivariate} is satisfied by all of our experiments.

In the case of noise-free outputs, a sample with $n=40$ observations was generated from $f_*$. The inputs, $\{ x_k \}$, followed a Student's $t$ distribution with $1$ degree of freedom, location $\mu = 0$ and scale $\gamma = 0.8$. Figure \ref{fig:experiment1} demonstrates that the MiNCE construction leads to informative confidence bands.

In the case of noisy observations, $n=400$ measurements were created. 
The $\{ x_k \}$ inputs followed a Student $t$ distribution with $1$ degree of freedom (Cauchy), location $\mu = 0$ and scale $\gamma = 0.5$. The measurement noises $\{\varepsilon_k\}$ followed a symmetric {\em Laplacian mixture} distribution, defined as follows: 
$$
\varepsilon_k 
\sim 
\begin{cases} 
\text{Laplace}\left(1, 0.5\right) & \text{if } Z_k = 1, \\[1mm]
\text{Laplace}\left(-\frac{1}{3}, 0.5\right) & \text{if } Z_k = 2,
\end{cases}
$$
where $\{Z_k\}$ are i.i.d.\ random variables with $\mathbb{P}(Z_k=1) = \frac{1}{4}$ and $\mathbb{P}(Z_k=2) = \frac{3}{4}$; 
$\text{Laplace}\left(\mu, \sigma^2\right)$ denotes the Laplace distribution with location $\mu$ and variance $\sigma^2$, i.e., scale $\sigma/\sqrt{2}$.
Observe that $\forall k :\mathbb{E}[\varepsilon_k] = 0$. This noise model was chosen to demonstrate that MiNCE can handle nontrivial distributions.

Figure \ref{fig:experiment2} shows the 
MiNCE 
bands for various confidence levels. The required confidence ellipsoid was built by the KGP method \cite{csaji2019distribution, csaji2023improving} with $n_0=20$.  As this mixture is not symmetric about zero, the permutation group was used, which is applicable, as $\{\varepsilon_k\}$ are i.i.d., hence, exchangeable. The norm bound was computed as in Remark \ref{remark-tau-computation}, 
$\tau \doteq \min\{\tau_0^{\circ}, \tau_n\}$. Since $\tau \leq \tau_n$, Lemma \ref{norm-consistency-2} ensures that it is asymptotically not larger than $\|f_*\|^2_\CH$, which can also be observed empirically: with $n_0 = \lceil\sqrt{n}\rceil$ and $\alpha = 0.1$, 
the median of $\tau$ over $50$ runs was $4.79$, $3.20$ and $2.92$, for $n = 100$, $250$ and $500$, respectively, while the median of the uncapped $\tau_0^{\circ}$ was $13.02$, $19.75$ and $51.02$, 
illustrating the need of capping with $\varrho$.

\begin{figure}[!t]
    \centering
	\hspace*{-2mm}	
    \includegraphics[width = 1\columnwidth]{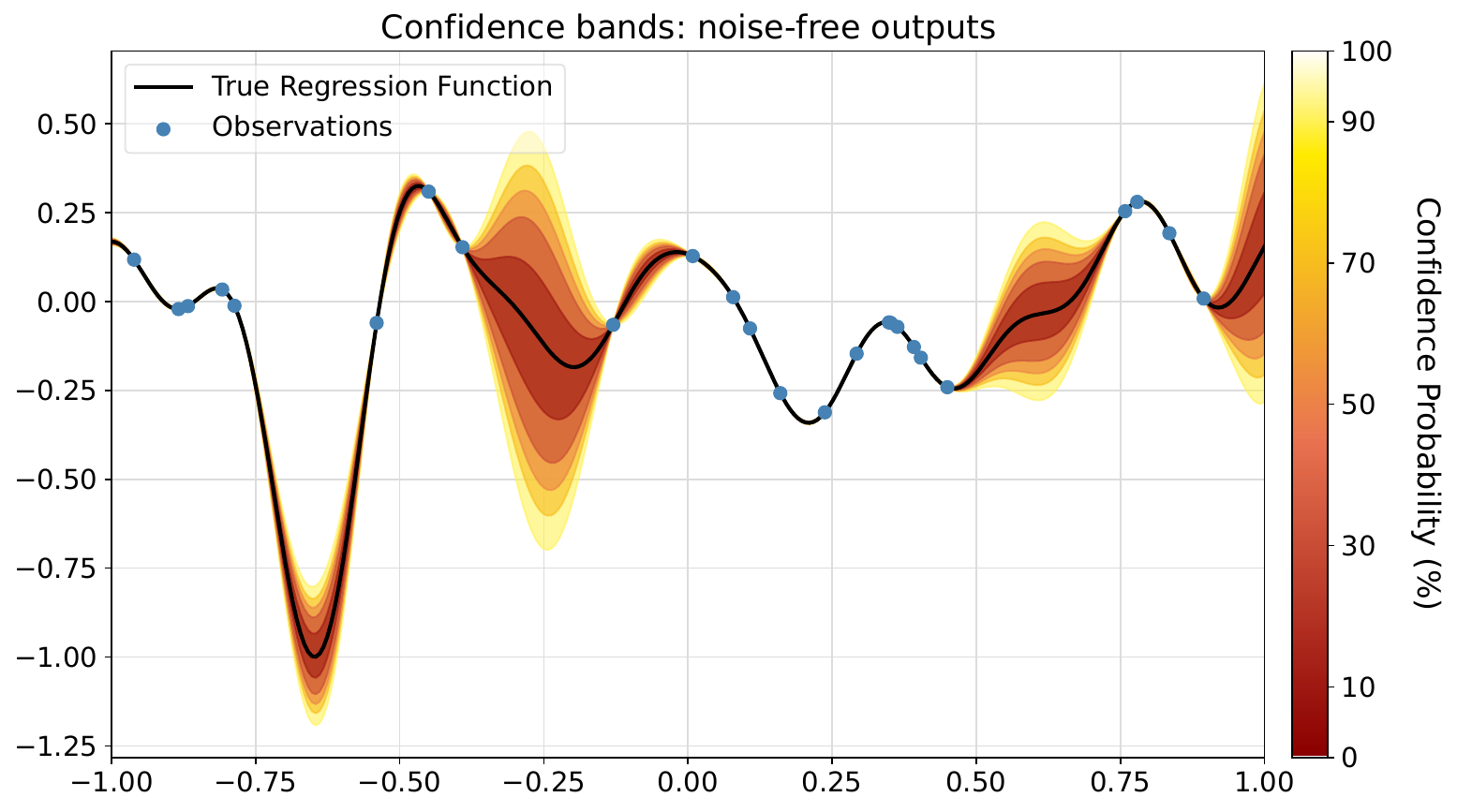} 	
    \caption{MiNCE confidence bands with levels $90\,\%$, $70\,\%$, $50\,\%$, $30\,\%$, and $10\,\%$, with $n=40$ noise-free observations and Student $t$ input distribution. The confidence levels can be interpreted by the color (scale) bar. MiNCE bands are increasing, i.e., $\CB_n^{\alpha_1} \subseteq \CB_n^{\alpha_2}$ if $\alpha_2 \leq \alpha_1$, where $\nu(\CB_n^{\alpha}) \geq 1- \alpha$.}
\label{fig:experiment1}
\end{figure}

We have also performed {\em quantitative} experiments regarding the {\em maximum errors} of the 
intervals, in the noise-free case, with a Paley--Wiener RKHS parameter $\eta = 30$, as follows.
\begin{enumerate}
    \item We have randomly generated $f_*$, which was then kept fixed, and, in each run, a new sample of inputs and (noise-free) outputs was created.
    \item An $x_0$ (query) input was chosen randomly from the distribution of the inputs (Student $t$ distribution with $1$ degree of freedom, location $\mu = 0$ and scale $\gamma = 0.8$).
    \item We have calculated the endpoints of the MiNCE confidence interval at $x_0$ for three different risk levels.
    \item The endpoints were intersected with the {\em a priori} interval induced by A\ref{assumption-Paley-Wiener-space-multivariate}, namely $|f_*(x_0)| \leq \sqrt{\varrho \, h_*(x_0)}$; if the resulting endpoints are $L_n(x_0)$ and $U_n(x_0)$, then the maximum error for query input $x_0$ is defined as $\max \{|L_n(x_0)-f_*(x_0)|, |U_n(x_0)-f_*(x_0)|\}$.
    \item This procedure was repeated $100$ times for each sample size $n$. The average and the standard deviation of the maximum errors were calculated.
\end{enumerate}

The results are summarized in Table \ref{table-diameters}. It can be observed that the mean maximum errors (and their standard deviations) shrink by more than an order of magnitude as the sample size increases, illustrating the strong
consistency of MiNCE.

Then, we have tested the case of estimating the smoothed spectrum. A sample with $n=10\,000$ noise-free observations was taken from $f_*$. The $\{x_k\}$ inputs had Student $t$ distribution with $1$ degree of freedom, $\mu=0$ and $\gamma=1$. The bandwidth parameter of the Paley--Wiener RKHS was set to $\eta = 30$.

The MiNCE confidence bands, $B_n^{\mathcal{M}}$ and $B_n^{\mathcal{P}}$, were constructed for the smoothed magnitude and the phase spectra. The triangular smoother was used with $h=1$. Figure \ref{fig:experiment5} and Figure \ref{fig:experiment6} show the results for different confidence levels, where the empirical Bernstein inequality was used to bound the norm of $f_*$, see 
\eqref{eq:emp-bernstein}. It can be observed that constructing confidence bands for the smoothed magnitude spectrum is an easier task than estimating the phase spectrum. On the other hand, both are much harder than estimating the signal in the spatial-domain, and require significantly more observations.
\begin{figure}[!t]
    \centering
	\hspace*{-2mm}	
    \includegraphics[width = 1\columnwidth]{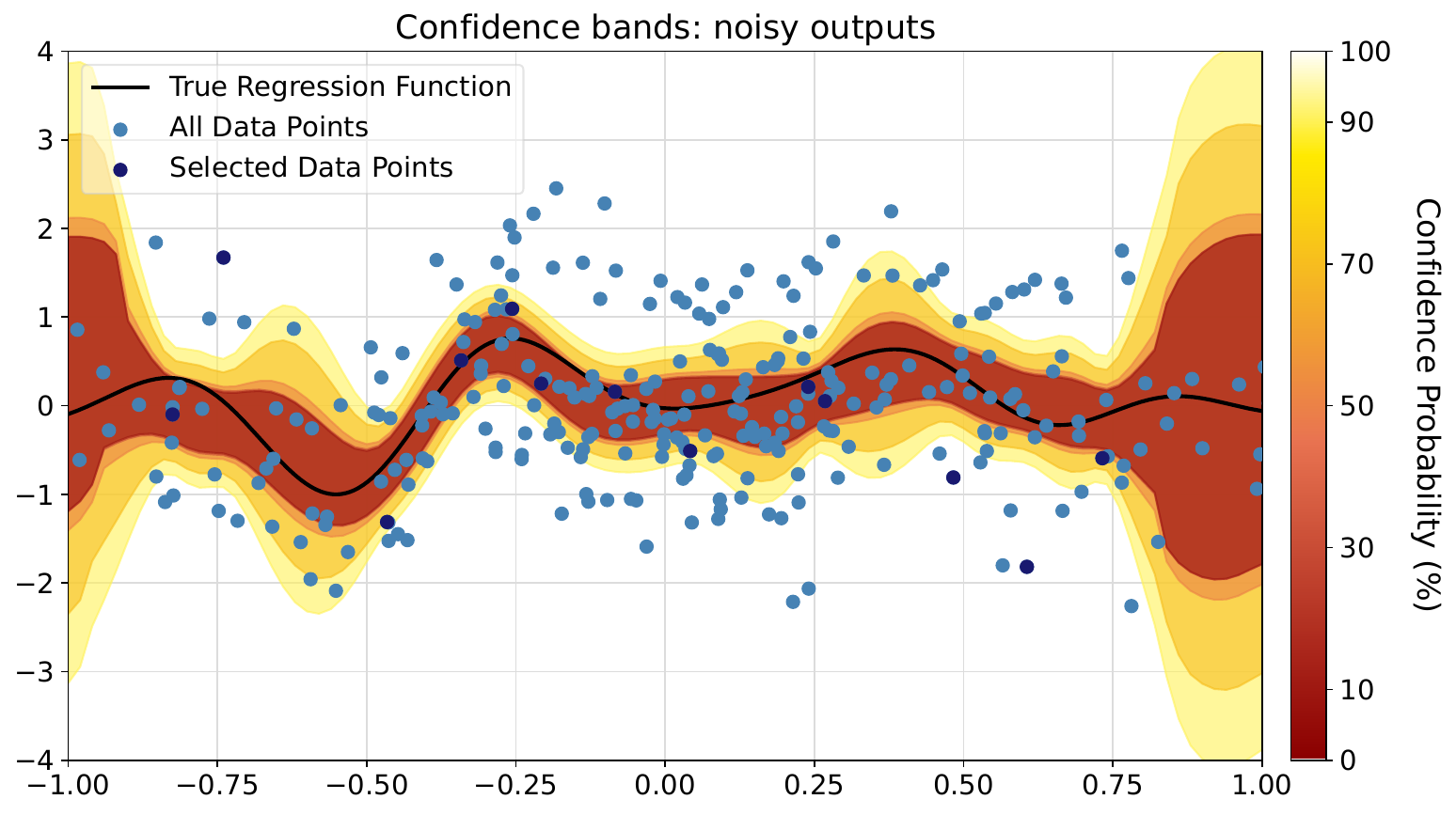} 	
    \caption{MiNCE confidence bands for various confidence levels between $90\,\%$ and $10\,\%$, see the color bar, with $n=400$ observations corrupted by Laplacian mixture measurement noises and having Student $t$ input distribution. The KGP confidence ellipsoid was created for the $n_0=20$ dark blue points.}
\label{fig:experiment2}
\end{figure}

{\renewcommand{\arraystretch}{1.3}

\begin{table}[!b]
\vspace*{-2mm}
\caption{Spatial-Domain: Averages and standard deviations of the\\[0mm]
maximum errors at random query inputs, in the\\[0mm]
noise-free case, for various risk levels.}\vspace{-2mm}
\centering
\begin{tabular}{|c|c||c|c|c|}
\cline{1-5}
{  Max error} & {  $n$} & {  $\alpha = 0.1$} & {  $\alpha = 0.5$} & {  $\alpha = 0.9$}\\ \hline\hline
avg (std) & 25  & 0.74 (0.56) & 0.71 (0.55) & 0.66 (0.53)\\ \hline
avg (std) & 50  & 0.21 (0.29) & 0.20 (0.27) & 0.18 (0.24)\\ \hline
avg (std) & 100 & 0.11 (0.19) & 0.10 (0.19) & 0.10 (0.17)\\ \hline
avg (std) & 200 & 0.05 (0.11) & 0.04 (0.10) & 0.04 (0.10)\\ \hline
avg (std) & 400 & 0.02 (0.07) & 0.02 (0.07) & 0.02 (0.06)\\ \hline
\end{tabular}
\label{table-diameters}
\vspace*{-3mm}
\end{table}
}

{\renewcommand{\arraystretch}{1.3}
\begin{table}[!b]
\vspace*{1.5mm}
\caption{Frequency-Domain: Averages and standard deviations of the relative
maximum errors of the magnitude spectrum bands for various sample sizes and risk levels.\vspace{-2mm}}
\centering
\begin{tabular}{|c|c|c||c||c|c|c|}
\hline
Rel. magn. & $n$ & $\alpha=0.1$ & $\alpha=0.5$ & $\alpha=0.9$\\ \hline\hline
avg (std) & 1000 & 2.20 (0.55) & 1.61 (0.42) & 1.28 (0.39)\\ \hline
avg (std) & 2000 & 1.67 (0.43) & 1.24 (0.36) & 1.00 (0.36)\\ \hline
avg (std) & 4000 & 1.24 (0.32) & 0.92 (0.28) & 0.74 (0.28)\\ \hline
avg (std) & 8000 & 0.97 (0.24) & 0.74 (0.24) & 0.61 (0.25)\\ \hline
avg (std) & 16000 & 0.75 (0.21) & 0.57 (0.20) & 0.47 (0.21)\\ \hline
\end{tabular}
\label{table-4}
\end{table}
}

Quantitative experiments were performed to evaluate the confidence bands for the {\em smoothed magnitude spectrum}, as well. The Paley--Wiener parameter was set to $\eta = 30$. The data generation followed the same ideas as in the case of the spatial-domain, but now the {\em relative} maximum errors were computed, defined for a query frequency $\omega_0$, which was chosen {\em uniformly} from $[-\eta,\eta]$, as $\max\{|L_n^{\mathcal{M}}(\omega_0)-\mathcal{M}_\phi(f_*)(\omega_0)|, |U_n^{\mathcal{M}}(\omega_0)-\mathcal{M}_\phi(f_*)(\omega_0)|\}/\max_{\omega \in [-\eta,\eta]} \mathcal{M}_\phi(f_*)(\omega)$, where $L_n^{\mathcal{M}}(\omega_0)$ and $U_n^{\mathcal{M}}(\omega_0)$ are the two endpoints of the 
interval $B_n^{\mathcal{M}}(\omega_0)$.

Table \ref{table-4} summarizes the results, where each row displays the averages and standard deviations of the 
maximum errors for $100$ random experiments, where a new $f_*$ and new observations were generated for each experiment. It demonstrates that the average maximum errors (and the variances) of the intervals decrease, as the sample size increases; though the required sample size is large, as the spectrum is simultaneously estimated on the whole domain, not just at a given frequency.
\vspace*{-4mm}

\begin{figure}[!t]
    \centering
	\hspace*{-2mm}	
	\includegraphics[width = 1\columnwidth]{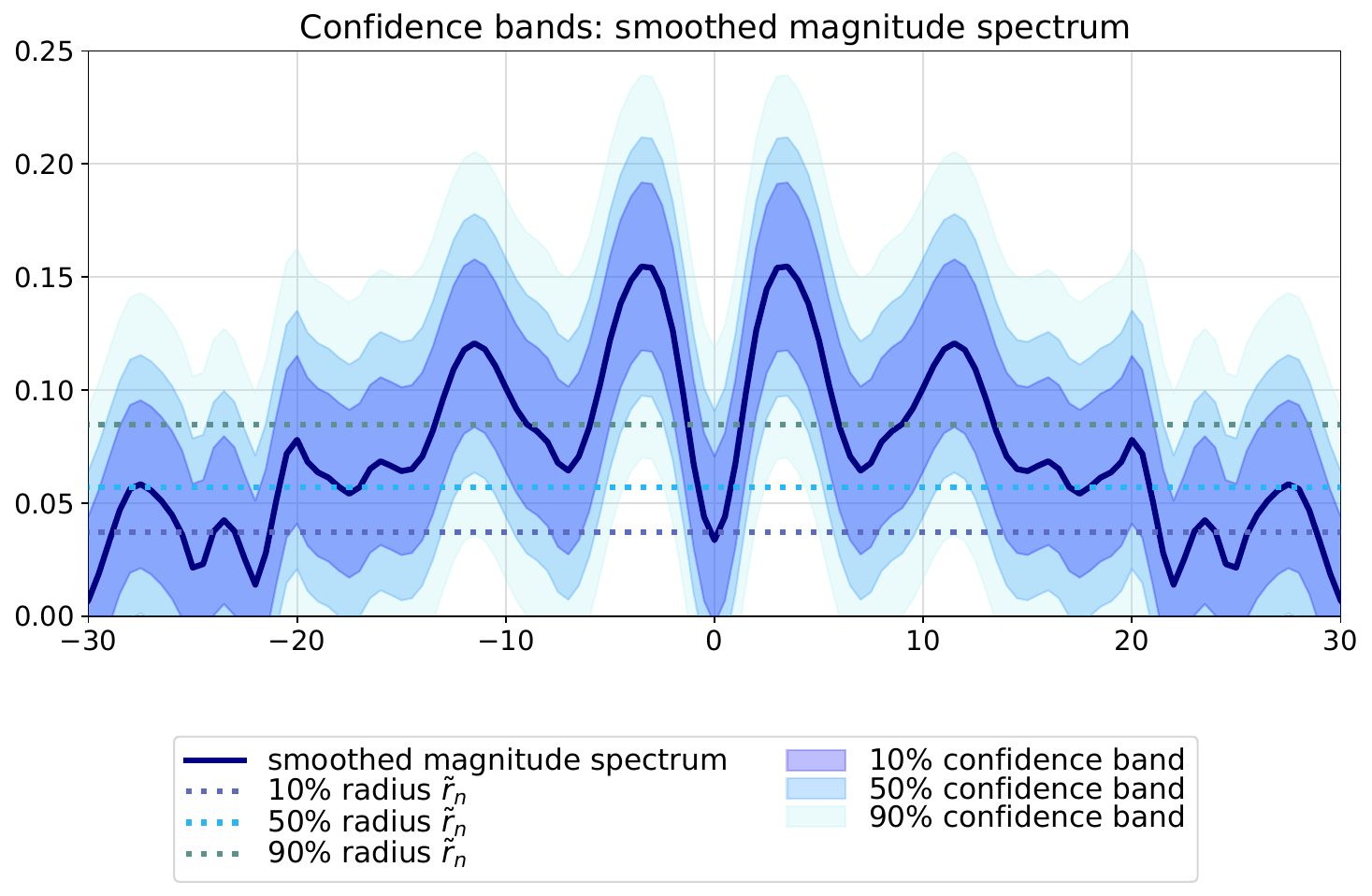} 	
    \caption{MiNCE confidence bands for the smoothed magnitude spectrum with $n=10\,000$ noise-free observations and using the triangular smoother.}
    \label{fig:experiment5}
\vspace*{-2mm}
\end{figure}

\section{Conclusions}
In this paper, we analyzed the {\em Minimum-Norm Confidence Envelope} (MiNCE) framework, a {\em nonparametric} approach to constructing {\em simultaneous} confidence bands with {\em nonasymptotic coverage} for {\em band-limited} functions from a finite sample of random input-output pairs, assuming a priori knowledge of the sampling distribution of the inputs. While the finite-sample coverage guarantees of these confidence envelopes were established previously, our main contribution was to show that they are also {\em strongly uniformly consistent}: as the sample size grows, the bands converge almost surely, uniformly over the entire domain, to the true regression function. We established this both for the noise-free case, where the target function is directly observed at random inputs, and for the noisy case, under the mild condition that a confidence ellipsoid for the (noiseless) outputs can be constructed at a subset of the observed sample inputs, a requirement satisfied, for example, by the Kernel Gradient Perturbation (KGP) method.

We further extended the MiNCE framework from the {\em spatial} to the {\em frequency} domain, building {\em nonparametric}, {\em simultaneous} confidence tubes for the {\em smoothed spectrum} of the target function, which induce bands for the magnitude and phase spectra. We proved that these regions inherit both the {\em nonasymptotic coverage} and 
the {\em strong uniform consistency} of the spatial regions. Experiments in both nonparametric regression and spectral estimation confirmed our 
results, illustrating the systematic contraction of the regions as $n$ grows.

A natural direction for future work is to relax the assumption that the sampling distribution of the inputs is known, and to complement consistency with convergence rates; both are directions we are actively pursuing. More broadly, since our consistency results rely primarily on a high-probability upper bound for the target function's kernel norm rather than on properties specific to Paley--Wiener spaces, extending the framework to other RKHS classes is also a promising avenue.

\begin{figure}[!t]
    \centering
	\hspace*{-2mm}	
	\includegraphics[width = 1\columnwidth]{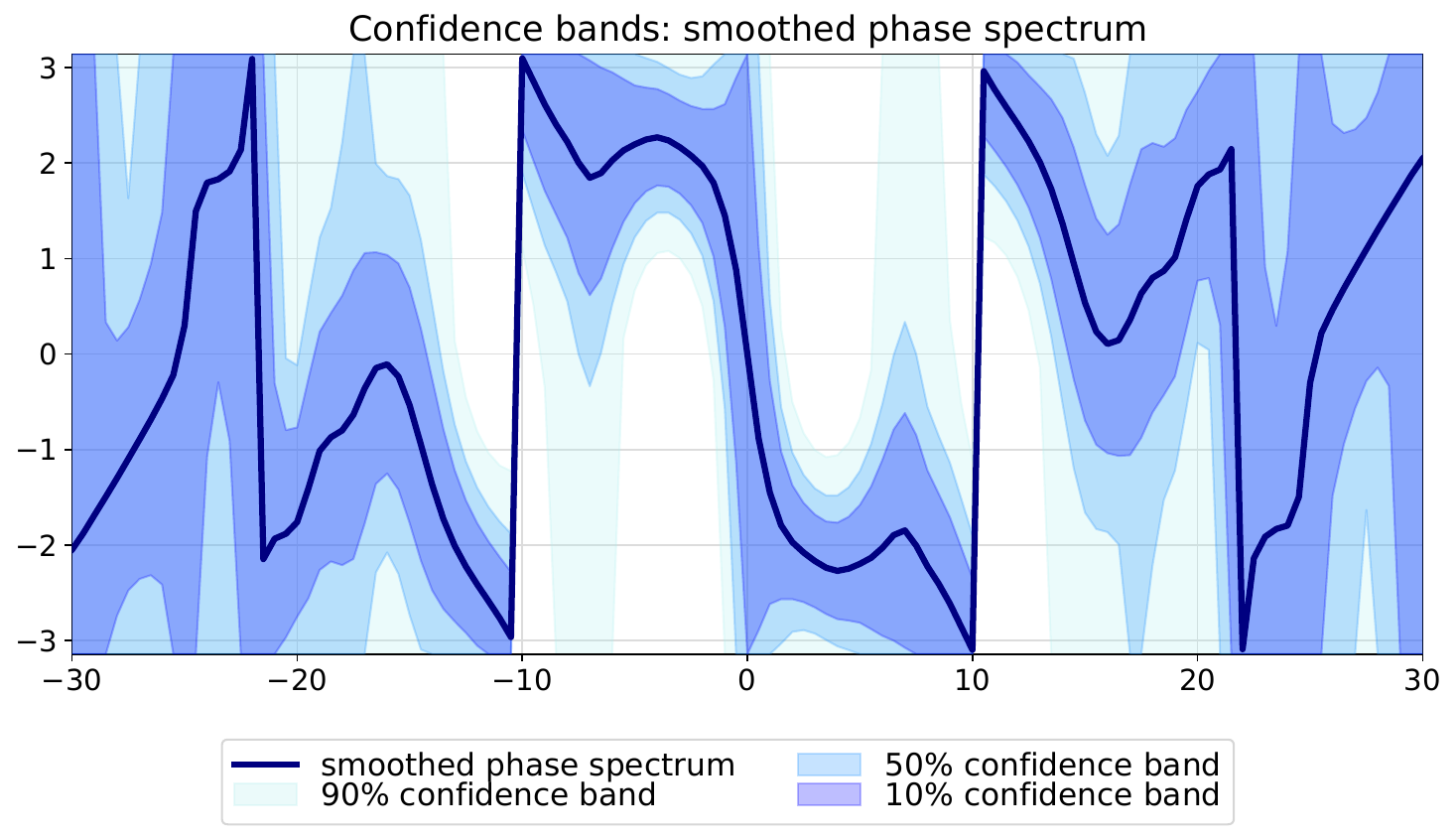} 	
    \vspace{-3.5mm}
    \caption{MiNCE confidence bands for the smoothed phase spectrum with $n=10\,000$ noise-free observations and using the triangular smoother. The phase band is bounded only where the smoothed magnitude (Figure \ref{fig:experiment5}) exceeds the corresponding radius $\tilde{r}_n$; below that threshold, the phase is unconstrained.}
\label{fig:experiment6}
\vspace*{-2mm}
\end{figure}

\vspace*{-1mm}
\bibliographystyle{ieeetr}
\bibliography{references}

\appendices
\section{Proof of Lemma \ref{lemma-norm-estimation-known-density-noisefree}}
\label{appendix-lemma1-proof}
\begin{proof}
Let us introduce the notation $$R \defeq \frac{1}{n} \sum_{k=1}^n \frac{y_k^2}{h_*(x_k)} = \frac{1}{n} \sum_{k=1}^n \frac{f_*^2(x_k)}{h_*(x_k)}.$$ We have
$\mathbb{E}[R] = \norm{f_*}_2^2.$ Then, from Hoeffding's inequality,
$$\mathbb{P}(R-\mathbb{E}[R] \leq -t) \leq \mbox{exp} (-2n t^2 / \varrho^2).$$
By using the complement rule, we obtain
$$\mathbb{P} (\mathbb{E}[R] < R+t) \geq 1-\mbox{exp}(-2n t^2 / \varrho^2).$$
We would like to choose a threshold $t>0$ such that
$$1-\alpha \leq \mathbb{P} (\mathbb{E}[R] < R+t).$$
The above inequality is satisfied, if we choose $t$ as
$$1-\alpha \leq 1-\mbox{exp}(-2n t^2 / \varrho^2)\, \Longrightarrow\, \mbox{exp}(-2n t^2 / \varrho^2) \leq \alpha.$$
By taking the natural logarithm on each side:
$$-2n t^2 / \varrho^2 \leq \mbox{ln}(\alpha)\, \Longrightarrow \,t^2 \geq \frac{\mbox{ln} (\alpha) \varrho^2}{-2n} \Rightarrow t\geq \varrho \cdot \sqrt{\frac{\mbox{ln} (\alpha)}{-2n}}.$$
By choosing $t^* \defeq \varrho \cdot \sqrt{{\mbox{ln} (1/\alpha)}/{(2n)}}$, we get that
$$\mathbb{P}(\norm{f_*}_\mathcal{H}^2 \geq R + t^*) \leq \alpha,$$
which completes the proof.
\end{proof}

\section{Proof of Theorem \ref{theorem-consistency-abstract-endpoint-based}}
\label{appendix-abst-to-band-consistency-proof}
\begin{proof}
Let us denote the bound on the kernel by $\zeta$, that is 
$$
\zeta\, \defeq\, \sup_{x\in\mathbb{R}^d}k(x,x)\,<\,\infty.
$$
For any $f\in\mathcal{H}$ and any $x\in\mathbb{R}^d$, by applying the reproducing property and the Cauchy-Schwarz inequality, we get
\begin{equation*}
\begin{aligned}
\|f\|_{\infty} &= \sup_{x\in\mathbb{R}^d}|f(x)|
= \sup_{x\in\mathbb{R}^d} \bigl|\bigl\langle f,\,k(\cdot,x)\bigr\rangle_{\mathcal{H}}\bigr|\\
&\le\,\|f\|_{\mathcal{H}}\;\|k(\cdot,x)\|_{\mathcal{H}}
=\|f\|_{\mathcal{H}}\;\sqrt{k(x,x)} = \sqrt{\zeta}\,\|f\|_{\mathcal{H}}.
\end{aligned}
\end{equation*}
This implies that, 
as $n\to\infty$, we have
$$
  \sup_{f\in\mathcal{C}_n}\|f - f_*\|_{\infty}
  \;\le\;\sqrt{\zeta}\sup_{f\in\mathcal{C}_n}\|f - f_*\|_{\mathcal{H}}
  \,\convergealmostsurely\;0.
$$
Further, for all $x \in \mathbb{R}^d$, by the construction of $L_n(x)$,
$$
  \bigl|L_n(x) - f_*(x)\bigr|
  = \bigl|\inf_{f\in\mathcal{C}_n}\! f(x) - f_*(x)\bigr|
  \le \sup_{f\in\mathcal{C}_n}\!|f(x)-f_*(x)|,
$$
and similarly for $U_n(x)$.  Taking ``$\sup$'' over all $x\in\mathbb{R}^d$ yields
$$
  \sup_{x\in\mathbb{R}^d}\bigl|L_n(x) - f_*(x)\bigr|
  \;\le\;\sup_{f\in\mathcal{C}_n}\!\|f-f_*\|_{\infty}
  \,\convergealmostsurely\;0,
$$
as $n\to \infty$, and similarly for $\sup_{x}|U_n(x) - f_*(x)|$.
\end{proof}

\section{Proof of Lemma \ref{parallel-perpendicular-lemma}}
\label{appendix-parallel-perpendicular-lemma}
\begin{proof}
First, we prove that for all $i \in [n]: f'(x_i) = f'_{\parallel}(x_i)$. By using the reproducing property of the RKHS, we have
\begin{equation*}
\begin{aligned}
f'(x_i) &=  \langle f', k(\cdot, x_i)\rangle_{\CH} = \langle f'_{\parallel}+f'_{\bot}, k(\cdot, x_i)\rangle_{\CH}\\
& = \underbrace{\langle f'_{\parallel}, k(\cdot, x_i)\rangle_{\CH}}_{f'_{\parallel}(x_i)} + \underbrace{\langle f'_{\bot}, k(\cdot, x_i)\rangle_{\CH}}_{0}\\
& = f'_{\parallel}(x_i),
\end{aligned}
\end{equation*}
where $\langle f'_{\bot}, k(\cdot, x_i)\rangle = 0$, since $k(\cdot, x_i) \in \mathcal{H}_n$.
Therefore, for all $f' \in \mathcal{C}^*_n$, we have $f'_{\parallel}(x_i) = f_*(x_i)$, for all $i \in [n]$. In other words, since $f'_{\parallel} \in \CH_n$, there are coefficients $\{\alpha_k\}$, such that\vspace{-1.5mm}
$$f_*(x_i) = f'_{\parallel}(x_i) = \sum_{k=1}^n \alpha_k k(x_i,x_k),$$
for $i \in [n]$. However, there is only one function in $\CH_n$ which can guarantee this property, namely $\hat{f}_n$. To see this, notice that the coefficients must satisfy
equation $K \alpha = y$, where $y\tr = [y_1, ...., y_n]$ and $K$ is the Gram matrix. Therefore, the unique solution is $\alpha = K^{-1} y$, since $K$ is (a.s.) invertible. This function, by its def., is the (min-norm) interpolant $\hat{f}_n$.
\end{proof}

\section{Proof of Lemma \ref{norm-consistency-2}}
\label{appendix-norm-consistency-2}
\begin{proof}
Bound $\tau_{n}$ is the sum of two terms:
\vspace{-1mm}
$$\tau_{n,1} \defeq \frac{1}{n_0} \sum_{k=1}^{n_0}\! \left(\frac{\max\{\nu_{n,k}^2, \mu_{n,k}^2\}}{h_*(x_k)} \land \varrho\right)\hspace{-0.8mm},\;
\tau_{n,2} \defeq \varrho \sqrt{\frac{\ln(1 / \alpha)}{2n_0}}.$$
We clearly have $\tau_{n,2} \to 0,$ since $n_0 \to \infty$ as $n \to \infty$. Thus it is enough to show $\tau_{n,1} \convergealmostsurely \kappa_*$.
Recall from \eqref{noise-free-density-convergence} that $\sigma_{n} \defeq \frac{1}{n_0}\sum_{k=1}^{n_0} f_*^2(x_k)/h_*(x_k) \convergealmostsurely \kappa_*$; note that, by A\ref{assumption-Paley-Wiener-space-multivariate}, in $\sigma_n$ the cap is inactive $(f_*^2(x_k)/h_*(x_k) \land \varrho = f_*^2(x_k)/h_*(x_k))$. Let
$$\delta_{n} \defeq \max_{k \in [n_0]} \max \{|\nu_{n,k} - f_*(x_k)|, |\mu_{n,k} - f_*(x_k)|\}.$$
For $z \in \ConfEll$, applying the reproducing property and the Cauchy-Schwarz inequality to $\hat{f}_{n, z-z_*}$, the minimum-norm interpolant of the residual $z - z_*$, gives $|z_k - f_*(x_k)|^2 = |\hat{f}_{n,z-z_*}(x_k)|^2 \leq k(x_k,x_k)\, (z-z_*)\tr K_{n_0}^{-1}(z-z_*)$. As $\nu_{n,k}$ and $\mu_{n,k}$ are attained on $\ConfEll$ and $\sup_x k(x,x) \leq \zeta < \infty$, we get $\delta_n^2 \leq \zeta \sup_{z \in \ConfEll}(z-z_*)\tr K_{n_0}^{-1}(z-z_*) \convergealmostsurely 0$, by \eqref{eq:A5-mahalanobis}.

By A\ref{assumption-Paley-Wiener-space-multivariate}, we have $|f_*(x_k)| \leq \sqrt{\varrho\, h_*(x_k)}$. Moreover, $|a-b| \leq \delta$ implies $|a^2 - b^2| \leq \delta(2|b|+\delta)$, the map $s \mapsto s \land \varrho$ is $1$-Lipschitz, and all terms are in $[\hspace{0.3mm}0, \varrho\hspace{0.3mm}]$. Hence, by the triangle inequality, $|\tau_{n,1}-\sigma_{n}|$ is at most the average of the absolute values of its $k$th terms, each of which is bounded by
$$\min\left\{\frac{2\sqrt{\varrho}\;\delta_{n}}{\sqrt{h_*(x_k)}} + \frac{\delta_{n}^2}{h_*(x_k)},\; \varrho\right\}\!.$$
Splitting the sum in $|\tau_{n,1}-\sigma_{n}|$ according to $h_*(x_k) \geq t$ or $h_*(x_k) < t$, we get for all $t > 0$,
$$|\tau_{n,1}-\sigma_{n}| \leq \frac{2\sqrt{\varrho}\,\delta_{n}}{\sqrt{t}} + \frac{\delta_{n}^2}{t} + \frac{\varrho}{n_0} \big|\{k \in [n_0] : h_*(x_k) < t\}\big|.$$
As $\delta_n \convergealmostsurely 0$ and, by the SLLN, the last term tends to $\varrho\, \mathbb{P}(h_*(x) < t)$, we have (a.s.) that $\limsup_n |\tau_{n,1}-\sigma_{n}| \leq \varrho\, \mathbb{P}(h_*(x) < t)$. Applying this for $t \defeq 1/m$, $m \in \mathbb{N}$, and using that $h_* > 0$ (A\ref{assumption-positive-input-density}), thus $\mathbb{P}(h_*(x) < 1/m) \to 0$, we obtain $|\tau_{n,1}-\sigma_{n}|\convergealmostsurely 0$. Finally, combining the results gives
$$
|\tau_n - \norm{f_*}_\CH^2| \leq |\tau_{n,1} - \sigma_{n}| + |\hspace{0.3mm}\sigma_{n}-\norm{f_*}_\CH^2|+\tau_{n,2}\convergealmostsurely 0,
$$
immediately implying that $\tau_n \convergealmostsurely \norm{f_*}_\CH^2$, as $n \to \infty$.\end{proof}

\section{Proof of Lemma \ref{ellipsoid-norm-diff}}
\label{appendix-ellipsoid-norm-diff}
\begin{proof}
Let $\CH_{n_0} \defeq \text{span}\{\hspace{0.3mm} k(\cdot,x_k): k \in [n_0] \hspace{0.3mm}\}$. Since $\CH_{n_0}$ is a finite dimensional subspace of $\CH$, it is closed, and we have
	\begin{equation*}
 			f = f_{\parallel} + f_{\bot}, \qquad \text{and} \qquad f' = f'_{\parallel} + f'_{\bot},
	\end{equation*}
by the Hilbert projection theorem, where $f_{\parallel}, f'_{\parallel} \in \CH_{n_0}$ are the minimum-norm interpolants of $z$ and $z'$, respectively (see the Proof of Lemma \ref{parallel-perpendicular-lemma} for  details),
and $f_{\bot}, f'_{\bot} \in \CH_{n_0}^{\bot}$. Then
 	\begin{equation*}
 		\|f_{\parallel}\|_\mathcal{H}^2 = z\tr K_{n_0}^{-1} z, \quad \text{and} \quad \|f'_{\parallel}\|_\mathcal{H}^2 = (z')\tr K_{n_0}^{-1} z',
 	\end{equation*}
as we discussed in Subsection \ref{sec:min-norm-int}. As $f_\parallel-f'_\parallel\in \CH_{n_0}$ and it is the unique function from $\CH_{n_0}$ that interpolates $z-z'$, it is the minimum-norm interpolant of $z-z'$ and hence 
\begin{equation}
    \label{eq:min-norm-dist}
	\|f_{\parallel} - f'_{\parallel}\|_\mathcal{H}^2 = (z - z')\tr K_{n_0}^{-1} (z - z').
 \end{equation}
Because $f_{\parallel} \perp f_{\bot}$ and $f'_{\parallel} \perp f'_{\bot}$, we get
 		\begin{align*}
 			\|f\|_\mathcal{H}^2 = \|f_{\parallel}\|_\mathcal{H}^2 + \|f_{\bot}\|_\mathcal{H}^2 &\;\Rightarrow\; \|f_{\bot}\|_\mathcal{H}^2 = \|f\|_\mathcal{H}^2 - z\tr K_{n_0}^{-1} z,
 		\end{align*}
and similarly for $\|f'\|_\mathcal{H}^2$, we have
$$
\|f'_{\bot}\|_\mathcal{H}^2 = \|f'\|_\mathcal{H}^2 - (z')\tr K_{n_0}^{-1} z'.
$$
By exploiting that $f \in \mathcal{D}^{*}_{n}$, we have $\|f\|_\mathcal{H}^2 \le \tau^*_n$, hence
\begin{equation}
\label{sec:f_bot_bound}
\|f_{\bot}\|_\mathcal{H} \le \sqrt{\tau^*_n - z\tr K_{n_0}^{-1} z},
\end{equation}
and a similar norm bound holds for $f'_{\bot}$, since we also have $\|f'\|_\mathcal{H}^2 \le \tau^*_n$. Combining \eqref{eq:min-norm-dist} and \eqref{sec:f_bot_bound} with the fact that
 \begin{align*}
 	\|f - f'\|_\mathcal{H}^2 &= \|f_{\parallel} - f'_{\parallel}\|_\mathcal{H}^2 + \|f_{\bot} - f'_{\bot}\|_\mathcal{H}^2\\
    &\leq  \|f_{\parallel} - f'_{\parallel}\|_\mathcal{H}^2 + (\|f_{\bot} \|_\mathcal{H} + \|f'_{\bot}\|_\mathcal{H})^2,
\end{align*}
implies the statement of the lemma.
\end{proof}

\section{Proof of Lemma \ref{ellipsoid-norm-convergence}}
\label{appendix-ellipsoid-norm-convergence}
\begin{proof}
By definition, $\ConfEllStar = \ConfEll \cup \{z_*\}$, therefore
\begin{equation*}
    \sup_{z \in \mathcal{E}_{n_0}^*} z\tr K_{n_0}^{-1} z = \max \Big\{ z_*\tr K_{n_0}^{-1} z_*, \, \sup_{z \in \mathcal{E}_{n_0}^n} z\tr K_{n_0}^{-1} z \Big\}.
\end{equation*}
We analyze the asymptotic behavior of both terms separately.

First, we consider the $z_*$ part. As $z_*\tr K_{n_0}^{-1} z_*$ is the kernel norm of the minimum-norm interpolant of the noiseless outputs, it converges (a.s.) to $\kappa_*$, as discussed in Remark \ref{remark-min-interp-norm-conv}.

Now, we consider the part with the confidence ellipsoid. Recall that $\tau_n$ is the solution of \eqref{noisy-norm-bound-with-intervals}.
Consider $z \in \mathcal{E}_{n_0}^n$,  then $z = z_* + (z - z_*)$. Expanding the quadratic form yields,
\begin{equation}
    \label{eq:exp-quadratic}
    z\tr K_{n_0}^{-1} z = 
\end{equation}
\begin{equation*}
    = z_*\tr K_{n_0}^{-1} z_* + 2 z_*\tr K_{n_0}^{-1} (z - z_*) + (z - z_*)\tr K_{n_0}^{-1} (z - z_*).
\end{equation*}
Then, by taking supremum over the confidence ellipsoid $\ConfEll$ and applying the triangle inequality, we get
\begin{equation*}
    \Big| \sup_{z \in \mathcal{E}_{n_0}^n}\! z\tr K_{n_0}^{-1} z - z_*\tr K_{n_0}^{-1} z_* \Big| \le  
    \vspace{-1mm}
\end{equation*}
\begin{equation*}
\leq 2\! \sup_{z \in \mathcal{E}_{n_0}^n}\! \big| z_*\tr K_{n_0}^{-1} (z - z_*) \big| +\! \sup_{z \in \mathcal{E}_{n_0}^n}\! (z - z_*)\tr K_{n_0}^{-1} (z - z_*).
\end{equation*}
By the Cauchy-Schwarz inequality, we have
\begin{equation*}
    \big| z_*\tr K_{n_0}^{-1} (z - z_*) \big| \le \sqrt{z_*\tr K_{n_0}^{-1} z_*} \sqrt{(z - z_*)\tr K_{n_0}^{-1} (z - z_*)}.
\end{equation*}
Using this and \eqref{eq:A5-mahalanobis}, we get
\begin{equation*}
    \mathbb{P}\bigg(\! \limsup_{n\to\infty} \sup_{z \in \mathcal{E}_{n_0}^n}\!\! \big| z_*\tr K_{n_0}^{-1} (z - z_*)\big| \leq \sqrt{\kappa_*} \cdot 0  = 0 \bigg) =\, 1,
\end{equation*}    
thus the cross-term vanishes (a.s.). Then, combining the results
\begin{equation*}
    \sup_{z \in \mathcal{E}_{n_0}^n}\! z\tr K_{n_0}^{-1} z \convergealmostsurely \kappa_*.
\end{equation*}

Since both arguments inside the maximum converge to $\kappa_*$, the overall supremum converges to $\kappa_*$, almost surely.

Furthermore, because $z_* \in \ConfEllStar$ by definition, we have 
$$
\inf_{z \in \ConfEllStar}\! z\tr K_{n_0}^{-1} z\leq
z_*\tr K_{n_0}^{-1} z_*.
$$
Then, by using again the decomposition of $z\tr K_{n_0}^{-1} z$ given by \eqref{eq:exp-quadratic}, dropping the nonnegative $(z - z_*)\tr K_{n_0}^{-1} (z - z_*)$ term and using the Cauchy-Schwarz inequality for the cross-product, but this time to get a lower bound, we get for all $z \in \mathcal{E}_{n_0}^*$:
\begin{equation*}
    z\tr K_{n_0}^{-1} z \ge z_*\tr K_{n_0}^{-1} z_* - 2 \sqrt{z_*\tr K_{n_0}^{-1} z_*\! \cdot\! (z - z_*)\tr K_{n_0}^{-1} (z - z_*)}.
\end{equation*}
Then, by taking infimum over $z\in\ConfEllStar$, we obtain
\begin{equation*}
    z_*\tr K_{n_0}^{-1} z_* - 2 \sqrt{z_*\tr K_{n_0}^{-1} z_*} \cdot d_n \le \inf_{z \in \mathcal{E}_{n_0}^*}\! z\tr K_{n_0}^{-1} z \le z_*\tr K_{n_0}^{-1} z_*,
    \vspace{-2mm}
\end{equation*}
where $d_n \doteq \sup_{z \in \mathcal{E}_{n_0}^*}\! \big((z - z_*)\tr K_{n_0}^{-1} (z - z_*)\big)^{\frac{1}{2}}$ represents the maximum Mahalanobis radius of the ellipsoid. 
Under A\ref{assumption-iid-multivariate}, A\ref{assumption-positive-input-density} and A\ref{assumption-Paley-Wiener-space-multivariate}, $z_*\tr K_{n_0}^{-1} z_* \convergealmostsurely \kappa_*$, see Remark \ref{remark-min-interp-norm-conv}. Further, under A\ref{assumption-ellipsoid-guarantee} and A\ref{assumption-ellipsoid-consistency},
$d_n \xrightarrow{\text{a.s.}} 0$, as $n \to \infty$, by \eqref{eq:A5-mahalanobis}.

Hence, we have sandwiched the infimum of $z\tr K_{n_0}^{-1} z$ over the extended ellipsoid $\ConfEllStar$ between two terms which both (a.s.) converge to $\kappa_*$, implying that the infimum also does.
\end{proof}

\end{document}

%% file: figtikz2.tex
\begin{tikzpicture}[scale=2, >=stealth]
\usetikzlibrary{calc}

    \pgfmathsetmacro{\xz}{4}
    \pgfmathsetmacro{\yz}{2}
    \pgfmathsetmacro{\r}{1}
    \pgfmathsetmacro{\squarelen}{0.3}
    
    \coordinate (O) at (0, 0);
    \coordinate (Z) at (\xz, \yz);

    \pgfmathsetmacro{\discriminant}{\xz*\xz + \yz*\yz - \r*\r} 
    \pgfmathsetmacro{\mdenom}{\xz*\xz - \r*\r} 
    \pgfmathsetmacro{\mnum}{\xz*\yz} 
    \pgfmathsetmacro{\msqrtterm}{\r*sqrt(\discriminant)} 
    
    \pgfmathsetmacro{\mone}{(\mnum + \msqrtterm) / \mdenom}
    \pgfmathsetmacro{\mtwo}{(\mnum - \msqrtterm) / \mdenom}
    
    \pgfmathsetmacro{\xtwo}{3.5 / \mone} 

    \pgfmathsetmacro{\phasezdeg}{atan(\yz/\xz)}
    \pgfmathsetmacro{\angletwodeg}{atan(\mtwo)}

    \pgfmathsetmacro{\angleperp}{\angletwodeg - 90}
    \coordinate (T) at ($(Z) + (\angleperp:\r)$);

    \fill[blue!15!white, dotted, line width=1.5pt] (Z) circle (\r);
    \draw[help lines, dashed, step=0.5, very thin, gray!50!white] (0, 0) grid (5.5, 3.5);
    \draw[->, line width=1.5pt] (0, 0) -- (5.5, 0);
    \draw[->, line width=1.5pt] (0, 0) -- (0, 3.5);

    \node at (2.75, 3.7) {Confidence ball at a query frequency with magnitude and phase bounds};

    \draw[line width=1pt, gray] (O) -- (5.5, {5.5*\mtwo}) node[right] {};
    \draw[line width=1pt, gray] (O) -- (\xtwo, 3.5) node[right] {};
    
    \draw[line width=1pt, purple!50!gray] (O) -- (Z);
    
    \pgfmathsetmacro{\distOZ}{sqrt(\xz*\xz + \yz*\yz)}
    \pgfmathsetmacro{\ratio}{\r / \distOZ}
    \coordinate (M_boundary) at ($(Z)!\ratio!(O)$);
    \draw[line width=3pt, dashed, blue!80!gray] (M_boundary) -- (Z);
    \node[rotate=\phasezdeg, anchor=south, inner sep=4pt] at ($(M_boundary)!0.5!(Z)$) {$\mathbf{\tilde{r}_n}$};
    
    \draw[line width=1pt,gray] (Z) -- (T);
    \fill[red!50!gray] (Z) circle (2.5pt);


    \node at (1.96, 0.72) {$\mathbf{\tilde{\gamma}_n}$};
    \node[purple!50!gray, rotate=\phasezdeg] at (2.3, 1.3) {magnitude};

    \node at (4.21, 2.22) {$\mathbf{(\mathcal{S}_{\phi}(\hat{f}_n))(\omega)}$};
    
    \node at (5.02, 2.71) {$B\hspace{0.3mm}'_{\hspace{-0.3mm}n}(\omega)$};
       
    \node at (3.8, {4.0*\mtwo - 0.25}) {$\mathbf{t_1}$};
    \node at (\xtwo - 0.5, {(\xtwo - 0.3)*\mone + 0.1}) {$\mathbf{t_2}$};
    \node[align=center] at (2.75, -0.15) {Real axis};
    \node[align=center, rotate= 90] at (-0.15, 1.75) {Imaginary axis};
    \coordinate (T1) at ($(T) - (\angletwodeg:\squarelen)$);
    \coordinate (T2) at ($(T) - (\angleperp:\squarelen)$);
    \coordinate (T_corner) at ($(T1) + (T2) - (T)$);
    \draw[line width=1pt, gray] (T1) -- (T_corner) -- (T2); 

\pgfmathsetmacro{\arcradius}{2.25} 
\pgfmathsetmacro{\gammadeg}{\phasezdeg - \angletwodeg} 
\draw[purple!50!gray, line width=1pt] (\angletwodeg:\arcradius) arc (\angletwodeg:\phasezdeg:\arcradius);
\draw[line width=3pt, dashed, blue!80!gray] (\angletwodeg:\arcradius) arc (\angletwodeg:\phasezdeg:\arcradius);
\draw[purple!50!gray, line width=1pt] (\angletwodeg:\arcradius) arc (\angletwodeg:0:\arcradius);
\node[purple!50!gray, rotate=0] at (2.48, 0.45) {phase};
\end{tikzpicture}